\documentclass[reqno]{amsart}

\usepackage{mathrsfs}
\usepackage{amscd}
\usepackage{verbatim}
\usepackage{mathtools}
\usepackage{xcolor}
\usepackage{enumerate}
\usepackage{bm}

\theoremstyle{plain}
\newtheorem{theorem}{Theorem}[section]
\theoremstyle{remark}
\newtheorem{remark}[theorem]{Remark}
\newtheorem{example}[theorem]{Example}

\theoremstyle{plain}

\newtheorem{lemma}[theorem]{Lemma}
\newtheorem{proposition}[theorem]{Proposition}
\newtheorem{definition}[theorem]{Definition}

\newtheorem{assumption}[theorem]{Assumption}

\def\avint_#1{\mathchoice%
	{\mathop{\kern 0.2em\vrule width 0.6em height 0.69678ex depth -0.58065ex
			\kern -0.8em \intop}\nolimits_{\kern -0.4em#1}}%
	{\mathop{\kern 0.1em\vrule width 0.5em height 0.69678ex depth -0.60387ex
			\kern -0.6em \intop}\nolimits_{#1}}%
	{\mathop{\kern 0.1em\vrule width 0.5em height 0.69678ex depth -0.60387ex
			\kern -0.6em \intop}\nolimits_{#1}}%
	{\mathop{\kern 0.1em\vrule width 0.5em height 0.69678ex depth -0.60387ex
			\kern -0.6em \intop}\nolimits_{#1}}}

\newcommand{\N}{\ensuremath{\mathbb{N}}}
\newcommand{\Z}{\ensuremath{\mathbb{Z}}}

\newcommand{\R}{\ensuremath{\mathbb{R}}}
\newcommand{\C}{\ensuremath{\mathbb{C}}}

\newcommand{\B}{\ensuremath{\mathcal{B}}}

\newcommand{\tr}{\text{tr}}

\newcommand{\osc}{{\rm osc}}

\DeclareMathOperator{\esssup}{ess.\,sup}

\def\Xint#1{\mathchoice
	{\XXint\displaystyle\textstyle{#1}}%
	{\XXint\textstyle\scriptstyle{#1}}%
	{\XXint\scriptstyle\scriptscriptstyle{#1}}%
	{\XXint\scriptscriptstyle%
		\scriptscriptstyle{#1}}%
	\!\int}
\def\XXint#1#2#3{{\setbox0=\hbox{$#1{#2#3}{%
				\int}$ }
		\vcenter{\hbox{$#2#3$ }}\kern-.6\wd0}}

\def\dashint{\Xint-}

\numberwithin{equation}{section}

\begin{document}
\numberwithin{equation}{section}

\author{Hongjie Dong}
\address{Division of Applied Mathematics\\
Brown University\\ Providence RI 02912, USA} \email{hongjie\_dong@bown.edu}
	
\author{Chiara Gallarati}
\address{Delft Institute of Applied Mathematics\\
	Delft University of Technology \\ P.O. Box 5031\\ 2600 GA Delft\\The
	Netherlands} \email{gallarati1@gmail.com}

\date\today
	
\title[Higher-order elliptic and parabolic equations]
{Higher-order elliptic and parabolic equations with VMO assumptions and general boundary conditions}
	
\begin{abstract}
We prove mixed $L_{p}(L_{q})$-estimates, with $p,q\in(1,\infty)$, for higher-order elliptic and parabolic equations on the half space $\R^{d+1}_{+}$ with general boundary conditions which satisfy the Lopatinskii--Shapiro condition.
We assume that the elliptic operators $A$ have leading coefficients which are in the class of vanishing mean oscillations both in the time variable and the space variable.
In the proof, we apply and extend the techniques developed by Krylov \cite{Kry07} as well as Dong and Kim in \cite{DK11} to produce mean oscillation estimates for equations on the half space with general boundary conditions.
\end{abstract}

\thanks{H. Dong was partially supported by the NSF under agreements DMS-1056737 and DMS-1600593. \\
	\indent C. Gallarati was supported by the Vrije Competitie subsidy 613.001.206 of the Netherlands Organisation for Scientific Research (NWO)}

\maketitle

\section{Introduction}
The $L_{p}(L_{q})$-regularity for differential equations has been proved to be a very useful tool for quasi-linear and nonlinear parabolic problems, as their solutions very often can be derived from the linear problem via elegant linearization techniques combined with the contraction mapping principle, see e.g. \cite{CleLi,Am,Lun}. For this, it is useful to look for minimal smoothness assumptions on the coefficients of the differential operators involved. Various approaches can be found in problems from mathematical physics, e.g. fluid dynamics, reaction-diffusion equations, material science, etc. See e.g. \cite{CleLi,GHHetc,MSWM}.

In this paper we establish $L_{p}(L_{q})$-estimates with $p,q\in(1,\infty)$ for higher-order parabolic equations of the form
\begin{equation}\label{prob:intro}
\begin{dcases}
u_t +(\lambda+A)u=f & {\rm on}\quad \R\times\R^{d}_{+}\\
{\tr}_{\R^{d-1}}B_{j}u=g_j & {\rm on}\quad \R\times\R^{d-1}, j=1,\ldots,m,
\end{dcases}
\end{equation}
where ``tr'' denotes the trace operator, $A$ is an elliptic differential operator of order $2m$, and $(B_j)$ is a family of differential operators of order $m_j<2m$ for $j=1,\ldots,m$. The coefficients of $A$ are assumed to be in the class of vanishing mean oscillations (VMO) both in the time and space variable, while the leading coefficients of $B_j$ are assumed to be constant. In addition, we assume that near the boundary $(A,B_j)$ satisfies the Lopatinskii--Shapiro condition. This condition was first introduced by Lopatinskii \cite{Lop53} and Shapiro \cite{Sha53}. See also the seminal work of Agmon--Douglis--Nirenberg \cite{ADN64}.
Roughly speaking, it is an algebraic condition involving the symbols of the principle part of the operators $A$ and $B_j$ with fixed coefficients, which is equivalent to the solvability of certain systems of ordinary differential equations.

Research on $L_{p}(L_{q})$-regularity for this kind of equations has been developed in the last decades by mainly two different approaches.
	
On the one hand, a PDE approach have been developed by a series of papers by Krylov, Dong, and Kim. Krylov in \cite{Krypq} showed $L_{p}(L_{q})$-regularity for second-order operators in the whole space with coefficients merely measurable in time and VMO in space, with the restriction $q\le p$.
The methodology of Krylov was then extended by Dong and Kim in \cite{DK09, DK11} to higher-order systems with the same class of coefficients. In \cite{DK11}, a new technique was developed to produce mean oscillation estimates for equations in the whole and half spaces with the Dirichlet boundary condition, for $p=q$. These results had been extended recently by the same authors in \cite{DK16} to mixed $L_p(L_q)$-spaces with Muckenhoupt weights and small BMO assumptions on the space variable, for any $p,q\in(1,\infty)$. It is worth noting that in all these references as well as others papers in the literature, VMO coefficients were only considered for equations with specific boundary conditions (Dirichlet, Neumann, or conormal, etc.).

On the other hand, from a functional analytic point of view, $L_{p}(L_{q})$-regularity can be viewed as an application of a more general abstract result, namely that of maximal $L_{p}$-regularity.
Maximal $L^p$-regularity means that, under certain assumption on $g_j$, for all $f\in L_p(\R,L_q(\R^d_+))$, the solution to the evolution problem \eqref{prob:intro} has the  ``maximal'' regularity in the sense that $u_t,  Au$ are both in $L_p(\R,L_q(\R^d_+))$.
In the case of time-independent coefficients, a complete operator-theoretic characterization of maximal $L_{p}$-regularity was introduced by Weis in \cite{Weis01}, using a new approach based on functional calculus and Fourier multiplier theorems. Using perturbation arguments combined with the characterization in \cite{Weis01}, one can study maximal $L_p$-regularity in the case when $t\mapsto A(t)$ is continuous. See, for instance, \cite{Am04,ACFP07,PS01}.
Recently, in \cite{GV,GVsystem} Gallarati and Veraar obtained maximal $L_p$-regularity for evolution equations with time-dependent operators, assuming only measurable dependence on time. This result was applied to show $L_{p}(L_{q})$-estimates for parabolic equations/systems in the whole space case in a weighted setting, for any $p,q\in(1,\infty)$, assuming that coefficients are uniformly continuous in the spatial variables and just measurable in the time variable. This generalized the results in \cite{Krypq}, where the restriction $q\le p$ is imposed, for this setting.
	
With coefficients in the class of VMO, higher-order systems in the whole space have been investigated in several papers, for example \cite{HH, HHH} where the leading coefficients are VMO with respect to the space variable and independent of the time variable, by using Muckenhoupt weights and estimates of integral operators of the Calder\'on--Zygmund type.
	
Concerning $L_p(L_q)$-regularity for equations on the half-space with boundary conditions satisfying the Lopatinskii--Shapiro condition, a breakthrough result was obtained by Denk, Hieber, and Pr\"{u}ss in \cite{DHP} in the case of autonomous initial boundary value problems with homogeneous boundary conditions and operator-valued constant coefficients. They combined operator sum methods with tools from vector-valued harmonic analysis to show $L_{p}(L_{q})$-regularity, for any $p,q\in(1,\infty)$, for parabolic problems with general boundary conditions of homogeneous type, in which the leading coefficients are assumed to be bounded and uniformly continuous.
Later, in \cite{DHP07}, the same authors characterized optimal $L_{p}(L_{q})$-regularity for non-autonomous, operator-valued parabolic initial-boundary value problems with inhomogeneous boundary data, where the dependence on time is assumed to be continuous. It is worth noting that in the special case of $m=1$, complex-valued coefficients and $q\le p$, a similar result was obtained by Weidemaier \cite{Weide02}.
The results of \cite{DHP07} have been generalized by Meyries and Schnaubelt in \cite{MS12b} to the weighted time-dependent setting, where the weights considered are Muckenhoupt power-type weights. See also \cite{MeyThesis}.
	
In this paper, we relax the assumptions on the coefficients of the operators involved. We obtain weighted $L_p(L_q)$-estimates for parameter-elliptic operators on the half space with coefficients VMO in the time and space variables, and with general boundary operators having constant leading coefficients and satisfying the Lopatinskii--Shapiro condition. An overview of our main result is given in the following theorem.

\begin{theorem}
Let $p,q\in(1,\infty)$. Then there exists $\lambda_0\geq 0$ such that for any $\lambda\geq\lambda_0$ and $u\in W^{1}_{p}(\R;L_{q}(\R^d_+))$
	$\cap L_{p}(\R;W^{2m}_{q}(\R^d_+))$ satisfying \eqref{prob:intro}, where $$
f\in L_p(\R;L_{q}(\R^d_+))\quad \text{and}\quad
g_j\in F^{k_j}_{p,q}(\R;L_{q}(\R^{d-1}))\cap L_{p}(\R;B_{q,q}^{2mk_j}(\R^{d-1}))
$$
with $k_j=1-m_j/(2m)-1/(2mq)$,  we have
\begin{multline*}
	\|u_t\|_{L_{p}(\R;L_{q}(\R^{d}_+))}+\sum_{|\alpha|\le 2m}\lambda^{1-\frac{|\alpha|}{2m}}\|D^{\alpha}u\|_{L_{p}(\R;L_{q}(\R^{d}_+))}\\
	\le C\|f\|_{L_{p}(\R;L_{q}(\R^{d}_+))} + C\|g_j\|_{F^{k_j}_{p,q}(\R;L_{q}(\R^{d-1}))\cap L_{p}(\R;B_{q,q}^{2mk_j}(\R^{d-1}))},
\end{multline*}
where $C>0$ is a constant independent of $\lambda$, $u$, $f$, and $g_j$.
\end{theorem}

This is stated in Theorem \ref{thm:VMOproblemLS}, where we also consider Muckenhoupt weights, and in the elliptic setting in Theorem \ref{thm:VMOellipticLS}.

To the best of our knowledge, these are the first results concerning equations with VMO coefficients and general boundary conditions. Our proofs are based on the results in \cite{DHP} combined with an  extension of the techniques developed in \cite{Kry07,krylov,DKBMO11,DK11,DK16}. In particular, in the main result of Section \ref{sec:moe}, Lemma \ref{lemma:moe1}, we prove mean oscillation estimates for equations on the half space with the Lopatinskii--Shapiro condition. A key ingredient of the proof is a Poincar\'e type inequality for solutions to equations satisfying the Lopatinskii--Shapiro condition, which is the main novelty of the paper.

To simplify the exposition and not to overburden this paper, here we only consider equations with boundary operators with constant leading coefficients.
In a subsequent work \cite{DG17II}, we will further study boundary operators with variable leading coefficients. In contrast to the case when $A$ has uniformly continuous leading coefficients, the extension of the results in this paper to boundary operators with variable leading coefficients is nontrivial and does not follow from the standard perturbation argument.
In fact, under the VMO assumption on the coefficients of $A$, in the case when the boundary operators have variable leading coefficients, to apply the method of freezing the coefficients as in Lemma \ref{lemma:moe2} below one would need to show the mean oscillation estimates of Lemma \ref{lemma:moe1} for an equation with inhomogeneous boundary conditions. To the best of the authors' knowledge, this case is not covered by the known theory. Moreover, the well-known localization procedure (see for instance \cite[Section 8]{DHP}) does not seem to directly apply to the case $p\neq q$, since we would need a partition of unity argument in both $t$ and $x$. The same problem would arise if one considers bounded smooth domains instead of the upper-half space: the technique of flattening the boundary would lead to an equation with boundary conditions with variable coefficients. This case will be treated as well in \cite{DG17II}.

The remaining part of the paper is organized as follows. In Section \ref{sec:preliminaries} we give the necessary preliminary results and introduce the notation. In Section \ref{sec:assumptions} we list the main assumptions on the operators and state the main results, Theorems \ref{thm:VMOproblemLS} and \ref{thm:VMOellipticLS}. In Section \ref{sec:moe} we prove the mean oscillation estimates needed for the proofs of the main theorems, which are given in Section \ref{sec:proofmainresult}. Finally, in Section \ref{sec:solvab} we prove a solvability result by using the a priori estimates in the previous sections.
	
\smallskip

{\em Acknowledgment.} – The authors would like to thank the anonymous referees for the careful reading and helpful comments.	
\section{Preliminaries}\label{sec:preliminaries}
In this section, we state some necessary preliminary results and introduce the notation used throughout paper.

\subsection{$A_p$-weights}\label{subsectionDG:weight}

Details on Muckenhoupt weights can be found in \cite[Chapter 9]{GrafakosModern} and \cite[Chapter V]{SteinHA}.

A {\em weight} is a locally integrable function on $\R^d$ with $w(x)\in (0,\infty)$ for almost every $x \in \R^d$. The space $L_p(\R^d,w)$ is defined as all measurable functions $f$ with
\begin{equation*}
	\|f\|_{L_p(\R^d,w)}=\Big(\int_{\R^d} |f|^p\ w \, dx\Big)^\frac{1}{p}<\infty \quad  \text{if $p\in [1, \infty)$},
\end{equation*}
and $\displaystyle \|f\|_{L_\infty(\R^d,w)} = \esssup_{x\in \R^d} |f(x)|$.

With this notion of weights and weighted $L_p$-spaces we can define the class of Muckenhoupt weights $A_{p}$ for all $p \in (1,\infty)$. A weight $w$ is said to be an {\em $A_{p}$-weight} if
\begin{align*}
	[w]_p=[w]_{A_{p}}:=\sup_{B} \Big(\dashint_B w(x) \, dx\Big) \Big(\dashint_B w(x)^{-\frac{1}{p-1}}\, dx \Big)^{p-1}<\infty.
\end{align*}
Here the supremum is taken over all balls $B\subset \R^d$ and $\avint_{B} = \frac{1}{|B|}\int_{B}$. The extended real number $[w]_{A_{p}}$ is called the {\em $A_p$-constant}. In the case of the half-space $\R^{d}_{+}$, we replace the balls $B$ in the definition by $B\cap \R^{d}_{+}=:B^{+}$ with center in $\overline{\R^{d}_{+}}$.

The classical Hardy--Littlewood maximal function theorem and the Fefferman--Stein theorem (see \cite[Theorem 9.1.9 and Corollary 7.4.6]{GrafakosModern}) have been recently generalized to mixed $L_{p}(\R,v;L_{q}(\R^d_+,w))$ spaces by Dong and Kim in Corollaries 2.6 and 2.7 of \cite{DK16}. Their proofs are based on the extrapolation theorem of Rubio de Francia (see \cite{Rubio82, Rubio83, Rubio84}, or  \cite[Chapter IV]{GarciaRubio}), that allows one to extrapolate from weighted $L_p$-estimates for a single $p\in (1,\infty)$ to weighted $L_q$-estimates for all $q\in (1,\infty)$.
These results will play an important role in the proof of Theorem \ref{thm:VMOproblemLS}, and thus we state them below for completeness.

{For $m=1,2,\ldots$ fixed depending on the order of the equations under consideration, we denote by
\begin{equation}\label{eq:paraboliccylynder}
Q_r^{+}(t,x)=((t-r^{2m},t)\times B_r(x))\cap \R^{d+1}_+
\end{equation}
the parabolic cylinders, where
$$
B_r(x)=\big\{y\in\R^d:|x-y|<r\big\}\subset\R^d
$$
denotes the ball of radius $r$ and center $x$.
We use $Q_r^{+}$ to indicate $Q_r^{+}(0,0)$. We also define
$$
B_r^+(x)=B_r(x)\cap \R^{d}_+.
$$
Let $\mathcal{Q}=\{Q_{r}^{+}(t,x):(t,x)\in\R^{d+1}_{+},\ r\in(0,\infty)\}$. Define for $p,q\in(1,\infty)$ the parabolic maximal function and sharp function of a function $f\in L_{p}(\R;L_{q}(\R^d_+))$ by
\[
\mathcal{M}f(t,x)=\sup_{\substack{Q\in\mathcal{Q}\\ (t,x)\in Q}}\dashint_{Q}|f(s,y)|\, dy\, ds
\]
and
\[
f^{\sharp}(t,x)=\sup_{\substack{Q\in\mathcal{Q}\\ (t,x)\in Q}}\dashint_{Q}|f(s,y)-\dashint_{\mathcal{Q}}f(t,x)\, dx\,dt|\, dy\, ds.
\]

\begin{theorem}[Corollary 2.6 of \cite{DK16}]\label{thm:HL}
	Let $p,q\in (1,\infty)$, $v\in A_{p}(\R)$ and $w\in A_{q}(\R^d_+)$. Then for any $f\in L_{p}(\R,v;L_{q}(\R^d_+,w))$, we have
	\[
	\|\mathcal{M}f\|_{L_{p}(\R,v;L_{q}(\R^d_+,w))}\le C\|f\|_{L_{p}(\R,v;L_{q}(\R^d_+,w))},
	\]
	where $C=C(d,p,q,[v]_{p},[w]_{q})>0$.
\end{theorem}

\begin{theorem}[Corollary 2.7 of \cite{DK16}]\label{thm:ffst}
	Let $p,q\in (1,\infty)$, $v\in A_{p}(\R)$ and $w\in A_{q}(\R^d_+)$. Then for any $f\in L_{p}(\R,v;L_{q}(\R^d_+,w))$, we have
	\[
	\|f\|_{L_{p}(\R,v;L_{q}(\R^d_+,w))}\le C\|f^{\sharp}\|_{L_{p}(\R,v;L_{q}(\R^d_+,w))},
	\]
	where $C=C(d,p,q,[v]_{p},[w]_{q})>0$.
\end{theorem}

\subsection{Function spaces and notation}\label{subsec:fsn}
In this section we introduce some function spaces and notation to be use throughout the paper.

We denote $D=-i(\partial_i,\ldots,\partial_d)$ and we consider the standard multi-index notation $D^{\alpha}=D_1^{\alpha_1}\cdot\ldots\cdot D_{d}^{\alpha_d}$ and $|\alpha|=\alpha_1+\cdots+\alpha_d$ for a multi-index $\alpha=(\alpha_1,\ldots,\alpha_d)\in\N_0^d$.

Denote
$$
\R^{d}_{+}=\big\{x=(x_1,x')\in\R^d: x_1> 0,\ x'\in\R^{d-1} \big\} \quad \text{and}\quad \R^{d+1}_{+}=\R\times\R^{d}_{+}.
$$
The parabolic distance between $X=(t,x)$ and $Y=(s,y)$ in $\R^{d+1}_{+}$ is defined by $\rho(X,Y)=|x-y|+|t-s|^{\frac{1}{2m}}$.
For a function $f$ on $\mathcal{D}\subset\R^{d+1}_{+}$, we set
\[
(f)_{\mathcal{D}}=\frac{1}{|\mathcal{D}|}\int_{\mathcal{D}}f(t,x)\, dx\,dt=\dashint_{\mathcal{D}}f(t,x)\, dx\,dt.
\]
Let $Q_r^{+}(t,x)$ be a parabolic cylinder as in \eqref{eq:paraboliccylynder}. We define the mean oscillation of $f$ on a parabolic cylinder as
\[
\osc(f,Q_r^{+}(t,x)):=\dashint_{Q_r^{+}(t,x)}
\big|f(s,y)-(f)_{Q_r^{+}(t,x)}\big|\,ds\,dy
\]
and we denote for $R\in(0,\infty)$,
\[
(f)^{\sharp}_{R}:=\sup_{(t,x)\in\R^{d+1}}\sup_{r\le R}\osc(f,Q_r^{+}(t,x)).
\]
Next, we introduce the function spaces which will be used in the paper.
For $p\in (1,\infty)$ and $k\in\N_0$, we define the standard Sobolev space as
$$
W^{k}_{p}(\R^{d}_{+})=\big\{u\in L_{p}(\R^{d}_{+}):\ D^{\alpha}u\in L_{p}(\R^{d}_{+})\quad \forall |\alpha|\le k \big\}.
$$
For $p,q\in(1,\infty)$, we denote
\[
L_{p}(\R^{d+1}_{+})=L_{p}(\R;L_{p}(\R^{d}_{+}))
\]
and mixed-norm spaces
\[
L_{p,q}(\R^{d+1}_{+})=L_{p}(\R;L_{q}(\R^{d}_{+})).
\]
For parabolic equations we denote for $k=1,2,\ldots$,
\[W^{1,k}_{p}(\R^{d+1}_{+})=W^{1}_{p}(\R;L_{p}(\R^{d}_{+}))\cap L_{p}(\R;W^{k}_{p}(\R^{d}_{+}))\]
and mixed-norm spaces
\[
W^{1,k}_{p,q}(\R^{d+1}_{+})=W^{1}_{p}(\R;L_{q}(\R^{d}_{+}))\cap L_{p}(\R;W^{k}_{q}(\R^{d}_{+})).
\]
	
We will use the following weighted Sobolev spaces. For $v\in A_{p}(\R)$ and $w\in A_q(\R^{d}_{+})$, we denote
\[
L_{p,q,v,w}(\R^{d+1}_{+})=L_{p}(\R,v;L_{q}(\R^{d}_{+},w))
\]
and
\[
W^{1,k}_{p,q,v,w}(\R^{d+1}_{+})=W^{1}_{p}(\R,v;L_{q}(\R^{d}_{+},w))\cap L_{p}(\R,v;W^{k}_{q}(\R^{d}_{+},w)),
\]
where by $f\in L_{p,q,v,w}(\R^{d+1}_{+})$ we mean
\[
\|f\|_{L_{p,q,v,w}(\R^{d+1}_{+})}:=	\bigg(\int_{\R}\bigg(\int_{\R^{d}_{+}}|f(t,x)|^{q}w(x)\,dx\bigg)^{p/q}v(t)\,dt\bigg)^{1/p}<\infty.
\]
\subsection{Interpolation and trace}
The following function spaces from the interpolation theory will be needed. For more information and proofs we refer the reader to \cite{MeyThesis,Tr92,Tr1}.
	
For $p\in (1,\infty)$ and $s=[s]+s_{\ast}\in \R_+\backslash \N_0$, where $[s]\in\N_0$, $s_{\ast}\in(0,1)$,  we define the Slobodetskii space $W^{s}_{p}$ by real interpolation as
\[
W^{s}_{p}=(W^{[s]}_{p},W^{[s]+1}_{p})_{s_{\ast},p}.
\]
For $m\in\N$ and $s\in (0,1]$ we consider anisotropic spaces of the form
\[
W^{s,2ms}_{p}(\R^{d+1}_{+})=W^{s}_{p}(\R;L_{p}(\R^{d}_{+}))\cap L_{p}(\R;W^{2ms}_{p}(\R^{d}_{+})).
\]

For $p\in (1,\infty)$, $q\in [1,\infty]$, $r\in\R$, and $X$ a Banach space, we introduce the Besov space $\B^{r}_{p,q}(\R^d)$ and the $X$-valued Triebel--Lizorkin space $F^{r}_{p,q}(\R^d,X)$ as defined below.

Let $\Phi(\R^d)$ be the set of all sequences $(\varphi_k)_{k\geq 0}\subset\mathcal{S}(\R^d)$ such that
\[
\widehat{\varphi}_0=\widehat{\varphi},\quad \widehat{\varphi}_1(\xi)=\widehat{\varphi}(\xi/2)-\widehat{\varphi}(\xi),\quad \widehat{\varphi}_k(\xi)=\widehat{\varphi}_{1}(2^{-k+1}\xi),
\]
where $k\geq 2$, $\xi\in\R^d$, and where the Fourier transform $\widehat{\varphi}$ of the generating function $\varphi\in\mathcal{S}(\R^d)$ satisfies $0\le \widehat{\varphi}(\xi)\le 1$ for $\xi\in\R^d$ and
\[
\widehat{\varphi}(\xi)=1\quad {\rm if}\ |\xi|\le 1,\quad  \widehat{\varphi}(\xi)=0\quad {\rm if}\ |\xi|\geq\frac{3}{2}.
\]
	
\begin{definition} Given $(\varphi_k)_{k\geq 0}\in \Phi(\R^d)$, we define the \textit{Besov space} as
\[
\B^{r}_{p,q}(\R^d)=\big\{ f\in \mathcal{S}'(\R^d):\ \|f\|_{\B^{r}_{p,q}(\R^d)}:=\|(2^{kr}\mathcal{F}^{-1}(\widehat{\varphi}_{k}\hat{f}))_{k\geq 0}\|_{\ell_q(L_{p}(\R^d))}  <\infty  \big\},
\]
and the \textit{$X$-valued Triebel--Lizorkin space} as
\begin{align*}
&F^{r}_{p,q}(\R^d,X)\\
&=\big\{ f\in \mathcal{S}'(\R^d,X):\ \|f\|_{F^{r}_{p,q}(\R^d,X)}:=\|(2^{kr}\mathcal{F}^{-1}(\widehat{\varphi}_{k}\hat{f}))_{k\geq 0}\|_{L_{p}(\R^d,\ell_q(X))} <\infty  \big\}.
\end{align*}
\end{definition}

Observe that $\B^{r}_{p,p}(\R^d)=F^{r}_{p,p}(\R^d)$ by Fubini's Theorem. Moreover, we have the following equivalent definition of Slobodetskii space
\[
W^{s}_{p}(\R^d)=
\begin{dcases}
W^{k}_{p}(\R^d), & s=k\in\N\\
\B^{s}_{p,p}(\R^d), & s\in \R_+\backslash \N_0.
\end{dcases}
\]
Later on we will consider $X$-valued Triebel-Lizorkin spaces on an interval $(-\infty,T)\subset\R$. We define these spaces by restriction.
\begin{definition}
Let $T\in(-\infty,\infty]$ and let $X$ be a Banach space. For $p\in(1,\infty)$, $q\in[1,\infty)$ and $r\in\R$ we denote by $F^{r}_{p,q}((-\infty,T);X)$ the collection of all restrictions of elements of $F^{r}_{p,q}(\R;X)$ on $(-\infty,T)$. If $f\in F^{r}_{p,q}((-\infty,T);X)$ then
$$
\|f\|_{F^{r}_{p,q}((-\infty,T);X)}=\inf\|g\|_{F^{r}_{p,q}(\R;X)}
$$
where the infimum is taken over all $g\in F^{r}_{p,q}(\R;X)$ whose restriction on $(-\infty,T)$ coincides with $f$.
\end{definition}		

The following spatial traces and interpolation inequalities will be needed in our proofs. For full details, we refer the reader respectively to \cite[Lemma 3.5 and Lemma 3.10]{DHP07}. See also \cite[Lemma 1.3.11 and Lemma 1.3.13]{MeyThesis}.

\begin{theorem}\label{thm:MeyLemma1.3.11}
Let $p\in(1,\infty)$, $m\in\N$, and $s\in (0,1]$ so that $2ms\in\N$. Then the map
\[
{\rm tr}_{x_1=0}:W^{s,2ms}_{p}(\R^{d+1}_{+})\hookrightarrow W^{s-\frac{1}{2mp},2ms-\frac{1}{p}}_{p}(\R\times\R^{d-1})
\]
is continuous.
\end{theorem}

\begin{lemma}\label{lemma:MeyLemma1.3.13}
Let $p\in(1,\infty)$ and let $m\in\N$ and $s\in[0,1)$ be given. Then for every $\varepsilon>0$, for $\beta\in\N_0^n$ with $s+\frac{|\beta|}{2m}+\frac{1}{2mp}<1$, it holds that for $u\in W^{1,2m}_{p}(\R\times\R^d_+)$,
\[
\|{\rm tr}_{\Omega}\nabla^{\beta}u\|_{W^{s,2ms}_{p}(\R\times\R^{d-1})}\le \varepsilon\|D^{2m}u\|_{L_{p}(\R\times\R^d_+)}+\varepsilon\|u_{t}\|_{L_{p}(\R\times\R^d_+)}+C_{\varepsilon}\|u\|_{L_{p}(\R\times\R^d_+)}.
\]
\end{lemma}
	
The following results for $p,q\in(1,\infty)$ will be important tools in the proof of Theorem \ref{thm:VMOproblemLS}.
	
\begin{theorem}\label{thm:spacialtracepq}
Let $p,q\in(1,\infty)$. Let for $j=1,\ldots,m$ and $m_j\in\{0,\ldots,2m-1\}$, $k_j=1-m_j/(2m) -1/(2mq)$. Then the  map
\begin{multline*}
{\rm tr}_{x_1=0}:W^{1-\frac{m_{j}}{2m}}_{p}(\R;L_q(\R^{d}_{+}))\cap L_p(\R;W^{2m-m_j}_{q}(\R^{d}_{+}))\\
\hookrightarrow F^{k_j}_{p,q}(\R;L_q(\R^{d-1}))\cap L_p(\R;\B^{2mk_j}_{q,q}(\R^{d-1}))
\end{multline*}
is continuous.
\end{theorem}
	
\begin{proof}
The proof is essentially contained in the proof of \cite[Proposition 6.4]{DHP07}, so we only give a sketched proof for the sake of completeness.
Let
\[u\in L_{p}(\R;W^{2m-m_{j}}_{q}(\R^{d}_{+})).\]
Taking traces in $x_1$ and applying \cite[Theorem 2.9.3]{Tr1} pointwise almost everywhere in time, we get  \[u|_{x_{1}=0} \in
L_{p}(\R;\B^{2m-m_{j}-\frac{1}{q}}_{q,q}(\R^{d-1})).\]
For the time regularity, let $u\in W^{1,2m}_{p,q}(\R\times\R^{d}_{+})$ and define $B$ as in \cite[Proposition 6.4]{DHP07} by
 $$B=(\partial_{t})^{\frac{1}{2m}}\quad {\rm with}\quad
D(B)=W^{\frac{1}{2m}}_{p}(\R;L_{q}(\R^{d}_{+})).
$$
Set $u_{j}=B^{2m-m_{j}-1}u$.
Then, $u_{j}\in W^{\frac{1}{2m}}_{p}(\R;L_{q}(\R^{d}_{+}))\cap L_{p}(\R;W^{1}_{q}(\R_{+};L_{q}(\R^{d-1})))$. Following the line of the proof of \cite[Proposition 6.4]{DHP07}, one can show that $u_{j}|_{x_1 =0}\in F^{\frac{1}{2m}-\frac{1}{2mq}}_{p,q}(\R;L_{q}(\R^{d-1}))$. This yields
\[D^{m_j}u|_{x_{1}=0}\in F^{k_j}_{p,q}(\R;L_{q}(\R^{d-1})),\]
which completes the proof.
\end{proof}

\begin{lemma}\label{lemma:interpolemmapq}
Let $p,q\in(1,\infty)$ and let $m\in\N$ and $s\in [0,1)$ be given. Then for every $\varepsilon>0$, for $\beta\in\N_0^n$ with $s+\frac{|\beta|}{2m}+\frac{1}{2mq}<1$, it holds that for $u\in W^{1,2m}_{p,q}(\R^{d+1}_{+})$,
\begin{multline*}
\|{\rm tr}_{\R^d_+}\nabla^{\beta}u\|_{F^{s}_{p,q}(\R;L_q(\R^{d-1}))\cap L_p(\R,v;\B^{2ms}_{q,q}(\R^{d-1}))}\\ \le\varepsilon\|D^{2m}u\|_{L_{p}(\R;L_{q}(\R^d_+))}+\varepsilon\|u_{t}\|_{L_{p}(\R;L_{q}(\R^d_+))}+C_{\varepsilon}\|u\|_{L_{p}(\R;L_{q}(\R^d_+))}.
\end{multline*}
\end{lemma}
The proof follows the line of \cite[Lemma 3.10]{DHP07}, by considering $p\neq q$ there and applying Theorem \ref{thm:spacialtracepq}.

\subsection{Anisotropic Sobolev embedding theorem}
We will use the following parabolic Sobolev embedding theorem. Details about the proof can be found in \cite[Section 18.12]{BesovVolII}.
	
We denote
\begin{multline}
W^{k,2m,h}_{t,x_{1},x';p}(\R^{d+1}_{+})
=W^{k}_{p}(\R;L_{p}(\R^{d}_{+}))\cap L_{p}(\R;W^{2m}_{p}(\R_+;L_{p}(\R^{d-1})))\\
\cap L_p(\R;L_p(\R_+;W_{p}^{h}(\R^{d-1}))).
\nonumber\end{multline}
	
\begin{theorem}\label{them:parabSobImb}
Let $p\in(1,\infty)$ and $m\in\N$. Then it holds for $k,h$ sufficiently large that
\[
W^{k,2m,h}_{t,x_{1},x';p}(Q_1^+)\hookrightarrow C^{\frac{2m-1/p}{2m},2m-1/p}(Q_1^+).
\]
Moreover,
\[
\|u\|_{C^{\frac{2m-1/p}{2m},2m-1/p}(Q_1^+)}\le C\|u\|_{W^{k,2m,h}_{t,x_{1},x';p}(Q_1^+)},
\]
with $C>0$ independent of $u$.
\end{theorem}
	
\section{Assumptions and main results}\label{sec:assumptions}
In this section let $p,q\in (1,\infty)$, $m=1,2,\ldots$ and we consider a $2m$-th order elliptic differential operator $A$ given by
\[
Au=\sum_{|\alpha|\le 2m}a_{\alpha}(t,x)D^{\alpha}u,
\]
where $a_\alpha:\R\times\R^{d}_{+}\rightarrow\C$. For $j=1,\ldots,m$ and $m_j \in\{0,\ldots,2m-1\}$, we consider the boundary differential operators $B_j$ of order $m_j$ given by
\[
B_j u=\sum_{|\beta|= m_j}b_{j\beta }D^{\beta}u + \sum_{|\beta|< m_j}b_{j\beta}(t,x)D^{\beta}u,
\]
where $b_{j\beta}\in\C$ if $|\beta|=m_j$, and $b_{j\beta}:\R\times\R^{d}_{+}\rightarrow\C$ if $\beta|<m_j$.

We will give conditions on the operators $A$ and $B_j$ under which there holds $L_p(L_q)$-estimates for the solution to the parabolic problem
\begin{equation}\label{prob:VMOtimespaceLS}
\begin{dcases}
u_t(t,x) + (A+\lambda)u(t,x)=f(t,x) & {\rm in}\ \R\times\R^{d}_{+}\\
B_{j}u(t,x)\big|_{x_1=0}=g_j(t,x) & {\rm on}\ \R\times\R^{d-1}\ \  j=1,\ldots,m,
\end{dcases}
\end{equation}	
and to the elliptic problem
\begin{equation}\label{prob:ellVMOtimespaceLS}
\begin{dcases}
(A+\lambda)u=f & {\rm in}\quad \R^{d}_{+}\\
B_{j}u\big|_{x_1=0}=g_j & {\rm on}\quad \R^{d-1},\ \  j=1,\ldots,m,
\end{dcases}
\end{equation}
where, for the elliptic case, the coefficients of the operators and data involved are functions independent on $t\in\R$, i.e., defined on $\R^d_+$.
\subsection{Assumptions on $A$ and $B_j$.}
We first introduce a parameter--ellipticity condition in the sense of \cite[Definition 5.1]{DHP}. Here $A_{\sharp}(t,x,\xi)=\sum_{|\alpha|=2m}a_{\alpha}(t,x)\xi^{\alpha}$ denotes the \textit{principal symbol} of the operator $A$.
\let\ALTERWERTA\theenumi
\let\ALTERWERTB\labelenumi
\def\theenumi{(E)$_\theta$}
\def\labelenumi{\textbf{(E)}$_\theta$}
\begin{enumerate}
\item\label{as:condell}
Let $\theta\in(0,\pi)$. For all $t\in\R$, $x\in\R^d_+$ it holds that
\begin{equation*}
A_{\sharp}(t,x,\xi)\subset\Sigma_{\theta},\quad  \forall\ \xi\in\R^{n},\ |\xi|=1,
\end{equation*}
where $\Sigma_{\theta}=\{z\in\C\backslash\{0\}:\ |\arg(z)|<\theta\}$ and $\arg : \C\backslash\{0\} \rightarrow (-\pi,\pi]$.
\end{enumerate}
\let\theenumi\ALTERWERTA
\let\labelenumi\ALTERWERTB

The following \ref{as:LScond}-condition is a condition of Lopatinskii--Shapiro type. Before stating it, we need to introduce some notation.

Denote by
$$
A_{\sharp}(t,x,D):=\sum_{|\alpha|=2m}a_{\alpha}(t,x)D^{\alpha}\quad
\text{and}\quad
B_{j,\sharp}(D):=\sum_{|\beta|=m_{j}}b_{j\beta}D^{\beta}
$$
the \textit{principal part} of $A(t,x)$ and $B_{j}$ respectively. Let $t_0\in\R$ and $x_0$ be in a neighborhood of $\partial\R^{d+1}_{+}$ of width $2 R_0$,  i.e., $x_0\in B_{2 R_0}(x')\cap\R^{d}_{+}$ for some $x'\in\partial\R^{d}_{+}$, and consider the operator $A_{\sharp}(t_0,x_0,D)$.  Taking the Fourier transform $\mathcal{F}_{x'}$  with respect to $x'\in\R^{d-1}$ and letting $v(x_1,\xi):=\mathcal{F}_{x'}(u(x_1,\cdot))(\xi)$, we obtain
\begin{align*}
A_{\sharp}(t_0,x_0,\xi,D_{x_1})v
&:=\mathcal{F}_{x'}(A_{\sharp}(t_0,x_0,D)u(x_1,\cdot))(\xi)\\
&= \sum_{k=0}^{2m}\sum_{|\beta|=k}a_{(\beta,k)}(t_0,x_0)\xi^{\beta}D_{x_{1}}^{2m-k}v
\end{align*}
and
\[
B_{j,\sharp}(\xi,D_{x_1})v:=\mathcal{F}_{x'}(B_{j,\sharp}(D)u(x_1,\cdot))(\xi)
=\sum_{k=0}^{m_j}\sum_{|\gamma|=k}b_{(\gamma,k)j}\xi^{\gamma}D_{x_{1}}^{m_j-k}v.
\]
where we denote $D_{x_{1}}=-i\frac{\partial}{\partial x_{1}}$.

\let\ALTERWERTA\theenumi
\let\ALTERWERTB\labelenumi
\def\theenumi{(LS)$_\theta$}
\def\labelenumi{\textbf{(LS)$_\theta$}}
\begin{enumerate}
\item\label{as:LScond}
Let $\theta\in(0,\pi)$ and let $t_0$ and $x_0$ be as above. For each $(h_{1},\ldots,h_{m})^{T}\in\R^{m}$ and each $\xi\in\R^{d-1}$ and $\displaystyle \lambda\in\overline{\Sigma}_{\pi-\theta}$, such that $|\xi|+|\lambda|\neq 0$, the ODE problem in $\R_+$
\begin{equation}\label{prob:defLS}
\begin{dcases}
\lambda v+ A_{\sharp}(t_0,x_0,\xi,D_{x_1})v=0,\quad  x_1 >0,\\
B_{j,\sharp}(\xi,D_{x_1})v\big|_{x_1=0}=h_{j},\quad  j=1,\ldots,m
\end{dcases}
\end{equation}
admits a unique solution $v\in C^{\infty}(\R_+)$ such that $\lim_{x\rightarrow \infty}v(x)=0$.
\end{enumerate}
\let\theenumi\ALTERWERTA
\let\labelenumi\ALTERWERTB	

\begin{remark}
In contrast to the original definition of the (LS)$_{\theta}$--condition (as for instance in \cite{DHP}), here we assume $x_0$ to be in a neighborhood of the boundary $\partial\R^{d+1}_{+}$ instead on the boundary itself. This choice is suitable to the VMO assumption on the coefficients of the operator $A$, which will be introduced in assumption \ref{as:operatorALS} below. 	
\end{remark}
We now introduce a regularity condition on the leading coefficients, where $\rho$ is a parameter to be specified.
\begin{assumption}[$\rho$] \label{ass:VMO}
	There exist a constant $R_0\in (0,1]$ such that $(a_{\alpha})^{\sharp}_{R_0}\le \rho$.
\end{assumption}	

Throughout the paper, we impose the following assumptions on the coefficients of $A$ and $B_j$.
\let\ALTERWERTA\theenumi
\let\ALTERWERTB\labelenumi
\def\theenumi{(A)}
\def\labelenumi{\textbf{(A)}}
\begin{enumerate}
	\item\label{as:operatorALS} The coefficients $a_{\alpha}$ are functions
	$\R\times\R^{d}_{+}\rightarrow\C$
	and satisfy Assumption \ref{ass:VMO} ($\rho$) with a parameter $\rho\in (0,1)$ to be determined later. Moreover there exists a constant $K>0$ such that $\|a_{\alpha}\|_{ L_{\infty}}\le K$, $|\alpha|\le 2m$, and there exists $\theta\in(0,\frac{\pi}{2})$ such that $A$ satisfies condition \ref{as:condell} .
\end{enumerate}
\let\theenumi\ALTERWERTA
\let\labelenumi\ALTERWERTB

\let\ALTERWERTA\theenumi
\let\ALTERWERTB\labelenumi
\def\theenumi{(B)}
\def\labelenumi{\textbf{(B)}}
\begin{enumerate}
	\item\label{as:operatorBLS} For each $j=1,\ldots, m$, the coefficients $b_{j\beta}$ are such that
	\begin{equation*}
	\begin{dcases}
	b_{j\beta}\in\C & {\rm if}\ |\beta|=m_j,\\
	b_{j\beta}:\R\times\R^{d}_{+}\rightarrow\C & {\rm if}\ |\beta|<m_j,
	\end{dcases}
	\end{equation*}
	and for $|\beta| < m_j$,  $b_{j\beta}\in C^{1-\frac{m_j}{2m},2m-m_{j}}(\R^{d+1}_{+})$ and there exists $K>0$ such that
	$$
	\|b_{j\beta}\|_{C^{1-\frac{m_j}{2m},2m-m_{j}}}\le K.
	$$
\end{enumerate}

\begin{remark}$ $
The \ref{as:LScond}-condition is essentially of algebraic nature, as it can be reformulated as a condition on the roots of a homogeneous polynomial. For further details, we refer the reader to \cite{WlokaBook} and \cite{RoitbergBook}.
It is not difficult to verify this condition in applications. For instance, see \cite[Section 3]{DPZ08} or \cite[Section 5.2]{MeyThesis}.
\end{remark}
	
\begin{example}
(i) Assume $A$ has order $2m$ and $B_j=D_{x_1}^{j-1}$, $j=1,\ldots,m$. Then, the Dirichlet boundary condition $B_j u|_{x_1=0}=g_j$ on $\partial\R^{d}_{+}$ satisfies the \ref{as:LScond}-condition. We refer the reader to \cite[Section I.2]{ADN64} for the proof. We remark that the complementing condition in \cite{ADN64} is equivalent to the \ref{as:LScond}-condition.

(ii) Let $A=\sum_{|\alpha|=2}a_{\alpha}D^{\alpha}$, with $a_{\alpha}\in\C$ and let $B=\sum_{|\beta|=1}b_{\beta}D^{\beta}$ with $0\neq b_{(1,0,\ldots,0)}\in\C$.
Then the \ref{as:LScond}-condition is equivalent to the algebraic condition that for each $\xi\in\R^{d-1}$ and $\lambda\in\overline{\Sigma}_{\pi-\theta}$ such that $|\xi|+|\lambda|\neq 0$, the characteristic polynomial
$$
a_{0}(\xi)\mu^{2}+a_{1}(\xi)\mu+a_{0}(\xi)+\lambda=0
$$
of \eqref{prob:defLS},  has two distinct roots $\mu_{\pm}$ with ${\rm Im}\mu_{+}>0>{\rm Im}\mu_{-}$, where $a_{k}(\xi)=\sum_{|\alpha|=k}a_{(k,\alpha)}\xi^{\alpha}$. The proof follows the line of \cite[Section 7.4]{KW}.
\end{example}

We can now state our main result.
\begin{theorem}\label{thm:VMOproblemLS}
Let $T\in(-\infty,\infty]$, $p,q\in(1,\infty)$. Let $v\in A_{p}((-\infty,T))$ and $w\in A_{q}(\R^{d}_{+})$. There exists
$$
\rho=\rho(\theta,m,d,K,p,q,[v]_{p},[w]_{q},b_{j\beta})\in (0,1)
$$
such that under the assumptions \ref{as:operatorALS}, \ref{as:operatorBLS}, and \ref{as:LScond} for some $\theta\in(0,\pi/2)$, the following hold.
		
(i) Assume the lower-order terms of $B_j$ to be all zero and $g_j\equiv 0$, with $j=1,\ldots,m$. Then there exists $\lambda_0=\lambda_0(\theta,m,d,K,p,q,R_0,[v]_{p},[w]_{q},b_{j\beta})\geq 0$ such that for any $\lambda\geq\lambda_0$ and
$$
u\in W^{1}_{p}((-\infty,T),v;L_{q}(\R^d_+,w))\cap L_{p}((-\infty,T),v;W^{2m}_{q}(\R^d_+,w))
$$
satisfying \eqref{prob:VMOtimespaceLS} on $(-\infty, T)\times \R^d_+$, where $f\in L_p((-\infty,T),v;L_{q}(\R^d_+,w))$, it holds that
\begin{multline}\label{eq:VMOtimespaceLS}
\|u_t\|_{L_{p}((-\infty,T),v;L_{q}(\R^{d}_+,w))}+\sum_{|\alpha|\le 2m}\lambda^{1-\frac{|\alpha|}{2m}}\|D^{\alpha}u\|_{L_{p}((-\infty,T),v;L_{q}(\R^{d}_+,w))}\\
\le C\|f\|_{L_{p}((-\infty,T),v;L_{q}(\R^{d}_+,w))},
\end{multline}
where $C=C(\theta,m,d,K,p,q,[v]_{p},[w]_{q},b_{j\beta})>0$ is a constant.
		
(ii) Let $v=w=1$. Then there exists $\lambda_0=\lambda_0(\theta,m,d,K,p,q,R_0,b_{j\beta})\geq 0$ such that for any $\lambda\geq\lambda_0$ and
$$
u\in W^{1}_{p}((-\infty,T);L_{q}(\R^d_+))\cap L_{p}( (-\infty,T);W^{2m}_{q}(\R^d_+))
$$
satisfying \eqref{prob:VMOtimespaceLS} on $(-\infty, T)$, where $f\in L_p((-\infty,T);L_{q}(\R^d_+))$ and
$$
g_j\in F^{k_j}_{p,q}((-\infty,T);L_{q}(\R^{d-1}))\cap L_{p}((-\infty,T);\B_{q,q}^{2mk_j}(\R^{d-1}))
$$
with $k_j=1-m_j/(2mq)-1/(2mq)$, it holds that
\begin{multline}\label{eq:VMOtimespaceLSgj}
\|u_t\|_{L_{p}((-\infty,T);L_{q}(\R^{d}_+))}+\sum_{|\alpha|\le 2m}\lambda^{1-\frac{|\alpha|}{2m}}\|D^{\alpha}u\|_{L_{p}((-\infty,T)
;L_{q}(\R^{d}_+))}\\
\le C\|f\|_{L_{p}((-\infty,T);L_{q}(\R^{d}_+))} + C\sum_{j=1}^{m}\|g_j\|_{F^{k_j}_{p,q}((-\infty,T);L_{q}(\R^{d-1}))\cap L_{p}((-\infty,T);\B_{q,q}^{2mk_j}(\R^{d-1}))},
\end{multline}
where $C=C(\theta,m,d,K,p,q,b_{j\beta})>0$ is a constant.
\end{theorem}
	
From the a priori estimates for the parabolic equation in Theorem \ref{thm:VMOproblemLS}, we obtain the a priori estimates for the higher-order elliptic equation as well, by using the arguments in \cite[Theorem 5.5]{DK16} and \cite[Theorem 2.6]{Kry07}. The key idea is that the solutions to elliptic equations can be viewed as steady state solutions to the corresponding parabolic cases.
	
We state below the elliptic version of Theorem \ref{thm:VMOproblemLS}. Here, the coefficients of $A$ and $B_j$ are now independent of $t$.

\begin{theorem}\label{thm:VMOellipticLS}
Let $q\in(1,\infty)$ and $w\in A_{q}(\R^{d}_{+})$. There exists
$$
\rho=\rho(\theta,m,d,K,q,[w]_{q})\in (0,1)
$$
such that under assumptions \ref{as:operatorALS}, \ref{as:operatorBLS}, and \ref{as:LScond} for some $\theta\in(0,\pi/2)$, the following hold.
		
(i) Assume the lower-order terms of $B_j$ to be all zero and consider homogeneous boundary conditions. Then, there exists $\lambda_0=\lambda_0(\theta,m,d,K,q,R_0,[v]_{q},b_{j\beta})\geq 0$ such that for any $\lambda\geq\lambda_0$ and $u\in W^{2m}_{q}(\R^d_+;w)$ satisfying \eqref{prob:ellVMOtimespaceLS}
where $f\in L_{q}(\R^d_+,w)$, it holds that
\begin{equation}\label{eq:VMOellipticLS}
\sum_{|\alpha|\le 2m}\lambda^{1-\frac{|\alpha|}{2m}}\|D^{\alpha}u\|_{L_{q}(\R^{d}_+,w)}\le C\|f\|_{L_{q}(\R^{d}_{+},w)},
\end{equation}
where $C=C(\theta,m,d,K,q,[w]_{q},b_{j\beta})>0$ is a constant.
		
(ii) Let $w=1$. Then there exists $\lambda_0=\lambda_0(\theta,m,d,K,q,R_0,b_{j\beta})\geq 0$ such that for any $\lambda\geq\lambda_0$ and $u\in W^{2m}_{q}(\R^d_+)$ satisfying
\begin{equation*}
\begin{dcases}
(A+\lambda)u=f & {\rm in}\ \R^{d}_{+}\\
B_{j}u\big|_{x_1=0}=g_j & {\rm on}\ \R^{d-1},
\end{dcases}
\end{equation*} where $f\in L_{q}(\R^d_+)$ and $g_j\in \B_{q,q}^{2mk_j}(\R^{d-1})$, with $k_j=1-m_j/(2m)-1/(2mq)$, it holds that
\begin{equation}\label{eq:VMOellipticLSgj}
\sum_{|\alpha|\le 2m}\lambda^{1-\frac{|\alpha|}{2m}}\|D^{\alpha}u\|_{L_{q}(\R^{d}_+)}\le C\|f\|_{L_{q}(\R^{d}_+)}+ C\sum_{j=1}^{m}\|g_j\|_{\B_{q,q}^{2mk_j}(\R^{d-1})},
\end{equation}
where $C=C(\theta,m,d,K,q,b_{j\beta})>0$ is a constant.
\end{theorem}

\begin{remark}$ $
(i) In Theorems \ref{thm:VMOproblemLS} and \ref{thm:VMOellipticLS} we focus only on the a priori estimates. The solvability of the corresponding equations will be derived in Section \ref{sec:solvab}.

(ii) For notational simplicity, in this paper we focus only on the scalar case. However, similar to \cite{DHP}, with the same proofs both Theorems \ref{thm:VMOproblemLS} and \ref{thm:VMOellipticLS} hold if one considers systems of operators, i.e., the coefficients $a_\alpha$ and $b_{j\beta}$ are $N\times N$ complex matrix-valued functions.

(iii) In  \cite{DHP, DHP07} and \cite{MS12b, MeyThesis}, the coefficients there considered are operator-valued, with values in a Banach space with the UMD property (Unconditional martingale difference, see \cite{HNVWbookI} for details). Since in our proofs we refer to these results when we freeze the coefficients and consider an unweighted setting, we believe that it is possible to extend our results also to the case of operator-valued coefficients, with values in a Hilbert space or in a UMD-Banach space. In particular, also the trace theorem needs to be extended to this case. Since in our results we do not include Muckenhoupt weights in the traces, this extension is straightforward by following \cite{MeyThesis}. For UMD-Banach valued coefficients in the weighted-space setting, we refer the reader to \cite{Lin17}.
\end{remark}

\section{Mean oscillation estimates for $u_{t}$ and $D^{\alpha}u$, $0\le|\alpha|\le 2m$,\\ except $D_{1}^{2m}u$}\label{sec:moe}
		
The main result of this section is stated in Lemma \ref{lemma:moe1}, and it shows mean oscillation estimates for $u_{t}$ and $D^{\alpha}u$, for all $0\le|\alpha|\le 2m$ except $D_{x_1}^{2m}u$. The proof of this lemma is the main novelty of the paper, and it generalizes some results in \cite{DK16} to general boundary conditions.
		
For a function $f$ defined on $\mathcal{D}\subset\R^{d+1}_{+}$, we set
\[
[f]_{C^{\frac{\nu}{2m},\nu}(\mathcal{D})}=\sup_{\substack{(t,x),(s,y)\in\mathcal{D}\\ (t,x)\neq(s,y)}}\frac{|f(t,x)-f(s,y)|}{|t-s|^{\frac{\nu}{2m}}+|x-y|^{\nu}}.
\]
Throughout the section, we assume that $A$ and $B_j$ consist only of their principal part.

Let
$$A_0=\sum_{|\alpha|=2m}\bar a_{\alpha}D^{\alpha}$$
be an operator with constant coefficients satisfying $|\overline{a}_{\alpha}|\le K$ for a constant $K>0$ and satisfying condition \ref{as:condell} with $\theta\in(0,\pi/2)$.
		
We first prove an auxiliary estimate, which is derived from a result in \cite{DHP07}.

\begin{lemma}\label{lemma:estconstcoeff}
Let $T\in(-\infty,+\infty]$ and $p,q\in(1,\infty)$. Let $A_0$ and $B_j$ be as above. Assume that for some $\theta\in(0,\pi/2)$ the \ref{as:LScond}-condition is satisfied.
Then for every $f\in L_{p,q}((-\infty,T)\times \R^{d}_{+})$ and
$$
g_j\in F^{k_j}_{p,q}((-\infty,T);L_{q}(\R^{d-1}))\cap L_{p}((-\infty,T);\B_{q,q}^{2mk_j}(\R^{d-1}))
$$
with $j\in\{1,\ldots,m\}$, $m_j \in\{0,\ldots,2m-1\}$, $k_j=1-m_j/(2m)-1/(2mq)$ and $u\in W^{1,2m}_{p,q}((-\infty,T)\times \R^{d}_{+})$ satisfying
\begin{equation}\label{prob:lemmaestconstcoeff}
\begin{dcases}
u_t(t,x) + (\lambda+A_0)u(t,x)=f(t,x) & {\rm in}\ (-\infty,T)\times\R^{d}_{+}\\
B_{j}u(t,x)\big|_{x_1=0}=g_j(t,x) & {\rm on}\ (-\infty,T)\times\R^{d-1},
\end{dcases}
\end{equation}
with $\lambda\ge 0$, we have
\begin{multline}\label{eq:lemmaestconstcoeff}
\|u_t\|_{L_{p,q}((-\infty,T)\times \R^{d}_{+})}+\sum_{|\alpha|\le 2m}\lambda^{1-\frac{|\alpha|}{2m}}\|D^{\alpha}u\|_{L_{p,q}((-\infty,T)\times \R^{d}_{+})}\\
\le C \|f\|_{L_{p,q}((-\infty,T)\times \R^{d}_{+})}
+ \sum_{j=1}^{m}\|g_j\|_{F^{k_j}_{p,q}((-\infty,T);L_{q}(\R^{d-1}))\cap L_{p}((-\infty,T);\B_{q,q}^{2mk_j}(\R^{d-1}))},
\end{multline}
with $C=C(\theta,m,d,K,p,q,b_{j\beta})>0$. Moreover, for any $\lambda>0$, $f\in L_{p,q}((-\infty,T)\times \R^{d}_{+})$ and
$$
g_j\in F^{k_j}_{p,q}((-\infty,T);L_{q}(\R^{d-1}))\cap L_{p}((-\infty,T);\B_{q,q}^{2mk_j}(\R^{d-1}))
$$
with $j\in\{1,\ldots,m\}$, $m_j \in\{0,\ldots,2m-1\}$, $k_j=1-m_j/(2m)-1/(2mq)$, there exists a unique solution $u\in W^{1,2m}_{p,q}((-\infty,T)\times \R^{d}_{+})$ to \eqref{prob:lemmaestconstcoeff}.
\end{lemma}
\begin{proof}
We divide the proof into several steps. First we assume that $T=\infty$.

{\em Step 1.} Let $u\in W^{1,2m}_{p,q}(\R_+\times \R^{d}_{+})$ be a solution to
\begin{equation}\label{prob:lemmadhpcosnt}
\begin{dcases}
u_t (t,x) + (\lambda+A_0)u(t,x)=f(t,x) & {\rm in}\ \R_+\times\R^{d}_{+}\\
B_{j}u(t,x)\big|_{x_1=0}=g_j(t,x) & {\rm on}\ \R_+\times\R^{d-1},\ \  j=1,\ldots,m\\
u(0,x)=0 &  {\rm on}\ \R^{d}_{+},
\end{dcases}
\end{equation}
with $\lambda>0$. By applying \cite[Proposition 6.4]{DHP07} to \eqref{prob:lemmadhpcosnt}, it holds that
\begin{equation}\label{eq:lemmadhpcosnt}
\begin{aligned}
&\|u_t\|_{L_{p,q}(\R_+\times \R^{d}_{+})}+\|D^{2m}u\|_{L_{p,q}(\R_+\times \R^{d}_{+})}\\
&\le C \|f\|_{L_{p,q}(\R_+\times \R^{d}_{+})}+ C\sum_{j=1}^{m}\|g_j\|_{F^{k_j}_{p,q}(\R_+;L_{q}(\R^{d-1}))\cap L_{p}(\R_+;\B_{q,q}^{2mk_j}(\R^{d-1}))},
\end{aligned}
\end{equation}
with $C=C(\lambda,\theta,m,d,K,p,q,b_{j\beta})$.
We remark that although the estimate is not explicitly stated in this reference, it can be extracted from the proofs there. We want to show that the estimate \eqref{eq:lemmadhpcosnt} also holds when $\lambda=0$.

For this, observe that in \cite[Proposition 6.4]{DHP07}, the coefficients of the operators under consideration are time and space dependent. In our case, since $A_0$ has constant coefficients, using a scaling $t\rightarrow \lambda^{-1} t$, $x\rightarrow \lambda^{-1/2m}x$, we obtain that the estimate \eqref{eq:lemmadhpcosnt} holds for any $\lambda\in(0,1)$ and with constant $C$ uniform in $\lambda$.
In fact, for a general $\lambda\in(0,1)$, let $v(t,x):=u(\lambda^{-1}t,\lambda^{-1/2m}x)$. Then $v$ satisfies
\begin{equation}\label{prob:lemmadhpscaling}
\begin{dcases}
v_t (t,x) + A_0 v(t,x) + v(t,x) =\tilde{f}(t,x) & {\rm in}\ \R_+\times\R^{d}_{+}\\
B_{j}v(t,x)\big|_{x_1=0}=\tilde{g}_{j}(t,x) & {\rm on}\ \R_+\times\R^{d-1}\\
v(0, x)=0 &  {\rm on}\ \R^{d}_{+}.
\end{dcases}
\end{equation}
where
$$
\tilde{f}(t,x)=\lambda^{-1}f(\lambda^{-1}t,\lambda^{-1/2m}x)
$$
and
$$
\tilde{g}_j(t,x)=\lambda^{-m_j/2m}g_{j}(\lambda^{-1}t,\lambda^{-1/2m}x).
$$
Applying \eqref{eq:lemmadhpcosnt} with $\lambda=1$ to \eqref{prob:lemmadhpscaling} we get that
\begin{equation}
\begin{aligned}
&\|v_t\|_{L_{p,q}(\R_+\times \R^{d}_{+})}+\|D^{2m}v\|_{L_{p,q}(\R_+\times \R^{d}_{+})}\\
&\le C \|\tilde{f}\|_{L_{p,q}(\R_+\times \R^{d}_{+})}+ C\sum_{j=1}^{m}\|\tilde{g}_j\|_{F^{k_j}_{p,q}(\R_+;L_{q}(\R^{d-1}))\cap L_{p}(\R_+;\B_{q,q}^{2mk_j}(\R^{d-1}))},
\end{aligned}
\end{equation}
with $C=C(\theta,m,d,K,p,q,b_{j\beta})$.
Now, scaling back and using the definition of the Besov space and Triebel--Lizorkin space, it is easily seen
\begin{equation}
\begin{aligned}
&\|u_t\|_{L_{p,q}(\R_+\times \R^{d}_{+})}+\|D^{2m}u\|_{L_{p,q}((0,\infty)\times \R^{d}_{+})}\\
&\le C \|f\|_{L_{p,q}(\R_+\times \R^{d}_{+})}+ C\sum_{j=1}^{m}\|g_j\|_{F^{k_j}_{p,q}(\R_+;L_{q}(\R^{d-1}))\cap L_{p}(\R_+;\B_{q,q}^{2mk_j}(\R^{d-1}))},
\end{aligned}
\end{equation}
where $C$ is independent of $\lambda\in (0,1)$.
Sending $\lambda\rightarrow 0$, we obtain that the estimate \eqref{eq:lemmadhpcosnt} holds when $\lambda=0$.
Finally, by applying a procedure of S. Agmon as in \cite[Theorem 4.1]{Kry07}, from \eqref{eq:lemmadhpcosnt} with $\lambda=0$ it follows that when $\lambda>0$,
\begin{equation}\label{eq:estagmon}
\begin{aligned}
&\|u_t\|_{L_{p,q}(\R_+\times \R^{d}_{+})}+\sum_{|\alpha|\le 2m}\lambda^{1-\frac{|\alpha|}{2m}}\|D^{\alpha}u\|_{L_{p,q}(\R_+\times \R^{d}_{+})}\\
&\le C \|f\|_{L_{p,q}(\R_+\times \R^{d}_{+})}+ C\sum_{j=1}^{m}\|g_j\|_{F^{k_j}_{p,q}(\R_+;L_{q}(\R^{d-1}))\cap L_{p}(\R_+;\B_{q,q}^{2mk_j}(\R^{d-1}))},
\end{aligned}
\end{equation}
with constant $C=C(\theta,m,d,K,p,q,b_{j\beta})$.

{\em Step 2.} Take $\eta\in C^{\infty}(\R)$ such that $\eta=1$ for $t>1$ and $\eta=0$ for $t<0$. Define $u_n=\eta(t+n)u$. From \eqref{prob:lemmaestconstcoeff}, we see that $u_n$ satisfies
\begin{equation}\label{eq:lemmaestconstcoeffproof}
\begin{dcases}
(u_n)_t (t,x) + (\lambda+A_0)u_n(t,x)=f_n(t,x) & {\rm in}\ (-n,\infty)\times\R^{d}_{+}\\
B_{j}u_n(t,x)\big|_{x_1=0}=g_{n,j}(t,x) & {\rm on}\ (-n,\infty)\times\R^{d-1}\\
u_n(-n,x)=0 &  {\rm on}\ \R^{d}_{+},
\end{dcases}
\end{equation}
for $j=1,\ldots,m$, where $\lambda>0$ and
$$
f_n=\eta(t+n)f+u\eta_t(t+n)\quad \text{and}\quad g_{n,j}(t,x)=\eta(t+n)g_j(t,x).
$$
By applying \eqref{eq:estagmon} to \eqref{eq:lemmaestconstcoeffproof}, we get that
\begin{align}\label{eq:proofscaling}
&\|(u_n)_t\|_{L_{p,q}((-n,\infty)\times \R^{d}_{+})}+\sum_{|\alpha|\le 2m}\lambda^{1-\frac{|\alpha|}{2m}}\|D^{\alpha}u_n\|_{L_{p,q}((-n,\infty)\times \R^{d}_{+})}\nonumber\\
&\le C \|f_n\|_{L_{p,q}((-n,\infty)\times \R^{d}_{+})}\nonumber\\
&\quad + C\sum_{j=1}^{m}\|g_{n,j}\|_{F^{k_j}_{p,q}((-n,\infty);L_{q}(\R^{d-1}))\cap L_{p}((-n,\infty);\B_{q,q}^{2mk_j}(\R^{d-1}))},
\end{align}
with $C=C(\theta,m,d,K,p,q,b_{j\beta})$.
Now, taking the limit as $n\rightarrow\infty$ yields \eqref{eq:lemmaestconstcoeff}, i.e., for any $\lambda>0$,
\begin{equation*}
\begin{aligned}
&\|u_t\|_{L_{p,q}(\R\times \R^{d}_{+})}+\sum_{|\alpha|\le 2m}\lambda^{1-\frac{|\alpha|}{2m}}\|D^{\alpha}u\|_{L_{p,q}(\R\times \R^{d}_{+})}\\
&\le C \|f\|_{L_{p,q}(\R\times \R^{d}_{+})}
+ C\sum_{j=1}^{m}\|g_j\|_{F^{k_j}_{p,q}(\R;L_{q}(\R^{d-1}))\cap L_{p}(\R;\B_{q,q}^{2mk_j}(\R^{d-1}))},
\end{aligned}
\end{equation*}
with $C=C(\theta,m,d,K,p,q,b_{j\beta})$.

{\em Step 3.} For the solvability, let $f\in L_{p,q}(\R^{d+1}_{+})$ and $g_j\in F^{k_j}_{p,q}(\R;L_{q}(\R^{d-1}))\cap L_{p}(\R;\B_{q,q}^{2mk_j}(\R^{d-1}))$, $j=1\ldots, m$. For integer $n>0$, define
$$
f_n=\eta(t+n)f\quad\text{and}\quad g_{n,j}=\eta(t+n)g_j
$$
so that
$f_n\rightarrow f$ in $L_{p,q}(\R^{d+1}_{+})$ and
$$
g_{n,j}\rightarrow g_j\quad \text{in}\,\, F^{k_j}_{p,q}(\R;L_{q}(\R^{d-1}))\cap L_{p}(\R;\B_{q,q}^{2mk_j}(\R^{d-1})).
$$
Now let $u_n\in W^{1,2m}_{p,q}((-n,\infty)\times \R^d_+)$ be the solution to the initial-boundary value problem with $f_n$ and $g_{n,j}$ and zero initial value at $t=-n$, the existence of which is guaranteed by \cite[Proposition 6.4]{DHP07}. We extend $u_n$ to be zero for $t<-n$. It is easily seen that $u_n$ satisfies \eqref{prob:lemmaestconstcoeff} with $f_n$ and $g_{n,j}$ in place of $f$ and $g_j$, respectively.
Applying the a priori estimate obtained in the argument above to $u_m-u_n$, we get that $\{u_n\}$ is a Cauchy sequence. Then the limit $u\in W^{1,2m}_{p,q}(\R^{d+1}_{+})$ is a solution to \eqref{prob:lemmaestconstcoeff}.

{\em Step 4.} For general $T<\infty$, we may assume $T=0$ by shifting the $t$-coordinate. We first take the even extensions of $u$ with respect to $t=0$.
Then $u\in W^{1,2m}_{p,q}(\R\times \R^{d}_{+})$. Next we take the even extension of $f$ and $g_j$ with respect to $t=0$. Let $v\in W^{1,2m}_{p,q}(\R\times \R^{d}_{+})$ be the solution to
\begin{equation*}
\begin{dcases}
v_t(t,x) + (\lambda+A_0)v(t,x)=f(t,x) & {\rm in}\ \R^{d+1}_{+}\\
B_{j}v(t,x)\big|_{x_1=0}=g_j(t,x) & {\rm on}\ \R\times\R^{d-1}, j=1,\ldots,m,
\end{dcases}
\end{equation*}
the existence of which is guaranteed by the argument above. Observe that $w:=u-v\in W^{1,2m}_{p,q}(\R\times \R^{d}_{+})$ satisfies
\begin{equation*}
\begin{dcases}
w_t(t,x) + (\lambda+A_0)w(t,x)=0 & {\rm in}\ (-\infty,0)\times\R^{d}_{+}\\
B_{j}w(t,x)\big|_{x_1=0}=0 & {\rm on}\ (-\infty,0)\times\R^{d-1}, j=1,\ldots,m.
\end{dcases}
\end{equation*}
We claim that $w=0$ on $t<0$. Indeed, for any $T_1<0$, we solve the equation of $w$ in $(T_1,\infty)\times \R^d_+$ with the zero initial data to get $w_1$, and extend $w_1=0$ for $t<T_1$.
It is easily seen that the extended function $w_1$ satisfies the same equation of $w$ in $\R\times \R^d_+$.
By the uniqueness of the solution, $w=w_1$. Therefore, $w=0$ when $t<T_1$ for any $T_1<0$. Then,
\begin{align*}
&\|u_t\|_{L_{p,q}((-\infty,0)\times \R^{d}_{+})}+\sum_{|\alpha|\le 2m}\lambda^{1-\frac{|\alpha|}{2m}}\|D^{\alpha}u\|_{L_{p,q}((-\infty,0)\times \R^{d}_{+})}\\
&=\|v_t\|_{L_{p,q}((-\infty,0)\times \R^{d}_{+})}+\lambda^{1-\frac{|\alpha|}{2m}}\|D^{\alpha}v\|_{L_{p,q}((-\infty,0)\times \R^{d}_{+})}
\\
&\le C \|f\|_{L_{p,q}(\R \times \R^{d}_{+})}
+ \sum_{j=1}^{m}\|g_j\|_{F^{k_j}_{p,q}(\R;L_{q}(\R^{d-1}))\cap L_{p}(\R;\B_{q,q}^{2mk_j}(\R^{d-1}))}\\
&= C \|f\|_{L_{p,q}((-\infty,0)\times \R^{d}_{+})}
+ \sum_{j=1}^{m}\|g_j\|_{F^{k_j}_{p,q}((-\infty,0);L_{q}(\R^{d-1}))\cap L_{p}((-\infty,T);\B_{q,q}^{2mk_j}(\R^{d-1}))}.
\end{align*}
The solvability is obtained by taking the even extension of $g_j$ and $f$, and then solve the equation in $\R\times \R^d_+$. The uniqueness follows from the a priori estimate.
\end{proof}

\begin{remark}
In Lemma \ref{lemma:estconstcoeff} as well as Theorem \ref{thm:VMOproblemLS}, we assumed $\theta\in(0,\pi/2)$. However, in \cite{DHP07, MeyThesis}, it is shown that in the case of operators with constant leading coefficients, or operators with uniformly continuous
leading coefficients in a bounded domain, it is sufficient that the conditions \ref{as:condell} and \ref{as:LScond} are satisfied for $\theta=\pi/2$, which are slightly weaker. The condition (E)$_{\pi/2}$ is also referred to as \textit{normal ellipticity} condition.
\end{remark}

From Lemma \ref{lemma:estconstcoeff}, we obtain the following H\"{o}lder estimate.
\begin{lemma}\label{lemma:iteration}
Let $0<r_{1}<r_{2}<\infty$. Let $v\in W^{1,2m}_{p}(Q_{r_2}^{+})$ be a solution to the homogeneous problem
\begin{equation}\label{prob:iterationlemma}
\begin{dcases}
v_{t} + A_0 v=0 & {\rm in}\ Q_{r_2}^+\\
B_j v\big|_{x_1=0}=0 & {\rm on}\ Q_{r_2}\cap\{x_1=0\},\ j=1,\ldots,m.
\end{dcases}
\end{equation}
Assume that for some $\theta\in(0,\pi/2)$ the \ref{as:LScond}-condition is satisfied.
Then there exists a constant $C=C(\theta,K,p,d,m,r_{1},r_{2},b_{j\beta})>0$ such that
\begin{equation}\label{eq:iteration}
\|v_{t}\|_{L_{p}(Q_{r_1}^{+})}+\|D^{2m}v\|_{L_p(Q_{r_1}^{+})}\le C\|v\|_{L_p(Q_{r_2}^{+})}.
\end{equation}
Furthermore, for $\nu=1-\frac{1}{p}$,
\begin{equation}\label{eq:estformoe}
[v_t]_{C^{\frac{\nu}{2m},\nu}(Q_{r_1}^+)}+[D^{2m-1}D_{x'}v]_{C^{\frac{\nu}{2m},\nu}(Q_{r_1}^+)}
\le C\|v_{t}\|_{L_{p}(Q_{r_2}^{+})}+C\|D^{2m}v\|_{L_{p}(Q_{r_2}^+)},
\end{equation}
with $C=C(\theta,K,p,d,m,r_{1},r_{2},b_{j\beta})>0$.
\end{lemma}
		
\begin{proof}
Set $R_0=r_1$ and $R_{i}=r_1+(r_2-r_1)(1-2^{-i})$, for $i=1,2,\ldots $. For each $i=0,1,2,\ldots$, take $\eta_{i}\in C_{0}^{\infty}(\overline{\R^{d+1}_{+}})$ satisfying
\begin{equation*}
\begin{dcases}
\eta_i=1 & {\rm in}\ Q_{R_i}^+\\
\eta_i=0 & {\rm outside}\ (-R_{i}^{2m},R_{i}^{2m})\times B_{R_{i+1}}
\end{dcases}
\end{equation*}
and
\begin{equation}\label{eq:moepropeta}
|D^{k}\eta_{i}|\le C2^{ki}(r_2-r_1)^{-k},\quad  |(\eta_{i})_{t}|\le C2^{2mi}(r_2-r_1)^{-2m}
\end{equation}
where $k=0,1,\ldots,2m$. It is easily seen that $v\eta_{i}\in W^{1,2m}_{p}(\R^{d+1}_{+})$ satisfies
\begin{equation}\label{prob:iteration}
\begin{dcases}
(v\eta_i)_{t}+A_0(v\eta_i)=f & {\rm in}\ \R^{d+1}_{+}\\
B_j(v\eta_i)\big|_{x_1=0}={\rm tr}_{x_1=0}G_j & {\rm on}\ \partial\R^{d+1}_{+},\ j=1,\ldots,m\\
(v\eta_i)(-r_2^{2m},\cdot)=0,
\end{dcases}
\end{equation}
where
\[
f=v(\eta_{i})_{t}+ \sum_{|\alpha|=2m}\sum_{|\gamma|\le 2m-1}\binom{\alpha}{\gamma} \bar a_{\alpha}D^{\gamma}vD^{\alpha-\gamma}\eta_{i}
\]
and
\[
G_j=\sum_{|\beta|=m_j}\sum_{|\tau|\le m_{j}-1}\binom{\beta}{\tau}b_{j\beta}D^{\tau}vD^{\beta-\tau}\eta_{i},\quad j=1,\ldots,m.
\]
Thus we extended \eqref{prob:iterationlemma} to a system on $\R\times\R^{d}_{+}$ without changing the value of $v$ on $Q_{r_1}^{+}$. Now let
$$
g_j={\rm tr}_{x_1=0}G_j\in W_{p}^{1-\frac{m_j}{2m}-\frac{1}{2mp},2m-m_{j}-\frac{1}{p}}(\R\times \R^{d-1}).
$$
By applying Lemma \ref{lemma:estconstcoeff} with $p=q$, we get
\begin{multline*}
\|(v\eta_{i})_{t}\|_{L_{p}(\R^{d+1}_{+})}+\|D^{2m}(v\eta_{i})\|_{L_{p}(\R^{d+1}_{+})}\\
\le C\|f\|_{L_{p}(\R^{d+1}_{+})}+C\sum_{j=1}^{m}\|g_{j}\|_{W_{p}^{1-\frac{m_j}{2m}-\frac{1}{2mp},2m-m_{j}-\frac{1}{p}}(\R\times \R^{d-1})},
\end{multline*}
where $C=C(\theta,K,d,m,p,b_{j\beta})$. By Theorem \ref{thm:MeyLemma1.3.11} with $s=1-\frac{m_{j}}{2m}\in (0,1]$, $m_{j}\in\{0,\ldots,2m-1\}$, we have
\[
\|g_{j}\|_{W_{p}^{1-\frac{m_j}{2m}-\frac{1}{2mp},2m-m_{j}-\frac{1}{p}}(\R\times \R^{d-1})}\le C\|G_{j}\|_{W_{p}^{1-\frac{m_j}{2m},2m-m_{j}}(\R^{d+1}_{+})}.
\]
Observe that
\[
\|f\|_{L_{p}(\R^{d+1}_{+})}\le C \|(\eta_{i})_{t}v\|_{L_{p}(\R^{d+1}_{+})}+C\sum_{|\alpha|=2m}
\sum_{|\gamma|\le 2m-1}\|D^{\gamma}vD^{\alpha-\gamma}\eta_{i}\|_{L_{p}(\R^{d+1}_{+})}
\]
and
\[
\|G_{j}\|_{W_{p}^{1-\frac{m_j}{2m},2m-m_{j}}(\R^{d+1}_{+})}\le C\sum_{|\beta|=m_j}\sum_{|\tau|\le m_j-1}\|D^{\tau}vD^{\beta-\tau}\eta_i\|_{W_{p}^{1-\frac{m_j}{2m},2m-m_{j}}(\R^{d+1}_{+})},
\]
where the constant $C=C(\theta,K,p,d,m)$ may vary from line to line.
By \eqref{eq:moepropeta}, it holds that
\[
\|(\eta_{i})_{t}v\|_{L_{p}(\R^{d+1}_{+})}\le C2^{2mi}(r_2-r_1)^{-2m}\|v\|_{L_{p}(Q_{r_2}^{+})}.
\]
By \eqref{eq:moepropeta} and interpolation inequalities (see e.g. \cite{krylov} and the proof of \cite[Lemma 3.2]{DKBMO11}), for $\varepsilon>0$ small enough and $|\gamma|\le 2m-1$ we get
\begin{equation*}
\begin{aligned}
&\|D^{\gamma}v D^{\alpha-\gamma} \eta_{i}\|_{L_p(\R^{d+1}_{+})}\le \|D^{\gamma}(v\eta_{i+1})D^{\alpha-\gamma}\eta_{i}\|_{L_{p}(\R^{d+1}_{+})}\\
&\le C2^{(2m-|\gamma|)i}(r_2-r_1)^{-(2m-|\gamma|)}\|D^{\gamma}(v\eta_{i+1})\|_{L_p(\R^{d+1}_{+})}\\
&\le \varepsilon \|D^{2m}(v\eta_{i+1})\|_{L_{p}(\R^{d+1}_{+})}
+C_{\varepsilon}2^{2mi}(r_2-r_1)^{-2m}\|v\|_{L_{p}(Q_{r_2}^{+})},
\end{aligned}
\end{equation*}
where $C_{\varepsilon}=C\varepsilon^{\frac{|\gamma|}{|\gamma|-2m}}$. Moreover, by the parabolic interpolation inequality
and \eqref{eq:moepropeta}, for $\varepsilon>0$ small enough and $|\tau|\le m_{j}-1$ we get
\begin{equation*}
\begin{aligned}
&\|D^{\tau}vD^{\beta-\tau}\eta_i\|_{W_{p}^{1-\frac{m_j}{2m},2m-m_{j}}(\R^{d+1}_{+})}\\
&\le C2^{(m_j-|\tau|)i}(r_2-r_1)^{-(m_j-|\tau|)}\|D^{\tau}(v\eta_{i+1})\|_{W_{p}^{1-\frac{m_j}{2m},2m-m_{j}}(\R^{d+1}_{+})}\\
&\le \varepsilon \|D^{2m}(v\eta_{i+1})\|_{L_{p}(\R^{d+1}_{+})} + \varepsilon \|(v\eta_{i+1})_{t}\|_{L_{p}(\R^{d+1}_{+})}
+C_{\varepsilon}2^{2mi}(r_2-r_1)^{-2m}\|v\|_{L_{p}(Q_{r_2}^{+})}
\end{aligned}
\end{equation*}
where $C_{\varepsilon}=C\varepsilon^{\frac{2m+|\tau|-m_j}{|\tau|-m_j}}$.
			
Combining the above inequalities yields
\begin{equation*}
\begin{aligned}
&\|(v\eta_{i})_{t}\|_{L_{p}(\R^{d+1}_{+})}
+\|D^{2m}(v\eta_{i})\|_{L_{p}(\R^{d+1}_{+})}
\le (C+C_{\varepsilon})2^{2mi}(r_2-r_1)^{-2m}\|v\|_{L_{p}(Q_{r_2}^{+})}\\
&+C\varepsilon\|D^{2m}(v\eta_{i+1})\|_{L_{p}(\R^{d+1}_{+})}+C\varepsilon\|(v\eta_{i+1})_{t}\|_{L_{p}(\R^{d+1}_{+})}.
\end{aligned}
\end{equation*}
We multiply both sides by $\varepsilon^{i}$ and we sum with respect to $i$ to get
\begin{align*}
&\sum_{i=0}^{\infty}\varepsilon^{i}\bigg(\|(v\eta_{i})_{t}\|_{L_{p}(\R^{d+1}_{+})}+\|D^{2m}(v\eta_{i})\|_{L_{p}(\R^{d+1}_{+})}\bigg)\\
&\le(C+C_{\varepsilon})(r_2-r_1)^{-2m}\sum_{i=0}^{\infty}(2^{2m}\varepsilon)^{i}\|v\|_{L_{p}(Q_{r_2}^{+})}\\
&\quad+C\sum_{i=1}^{\infty}\varepsilon^{i}\bigg(\|D^{2m}(v\eta_{i})\|_{L_{p}(\R^{d+1}_{+})}
+\|(v\eta_{i})_{t}\|_{L_{p}(\R^{d+1}_{+})}\bigg).
\end{align*}
We choose $\varepsilon=2^{-2m-1}$ and observe that the above summations are finite. Then, the above estimate gives
\begin{equation}\label{eq:moefinalestimate}
\|(v\eta_{0})_{t}\|_{L_{p}(\R^{d+1}_{+})}+\|D^{2m}(v\eta_{0})\|_{L_{p}(\R^{d+1}_{+})}\le C (r_2-r_1)^{-2m}\|v\|_{L_{p}(Q_{r_2}^{+})}.
\end{equation}
Since the left-hand side of \eqref{eq:moefinalestimate} is greater than that of \eqref{eq:iteration}, we can conclude
\[
\|v_{t}\|_{L_{p}(Q_{r_1}^{+})}+\|D^{2m}v\|_{L_p(Q_{r_1}^{+})}\le C(r_2-r_2)^{-2m}\|v\|_{L_{p}(Q_{r_2}^{+})},
\]
with $C=C(\theta,K,p,d,m,b_{j\beta})$.
		
To show the H\"{o}lder estimate for $v$, we proceed as follows. First, observe that from \eqref{eq:iteration} and interpolation inequalities, it holds that
\begin{equation}\label{eq:iterationlemmaproof}
\|v\|_{W^{1,2m}_{p}(Q^{+}_{r_1})}\le  C\|v\|_{L_{p}(Q_{r_2}^+)}.
\end{equation}
Observe now that for $k,h>0$, the derivatives $D^{k}_{t}D^{h}_{x'}v$ satisfy the same equation as $v$. Hence, from \eqref{eq:iterationlemmaproof} and a standard bootstrap argument, it holds that $v\in W^{k,2m,h+2m}_{t,x_{1},x';p}(Q^{+}_{r_1})$ with
\[
\|v\|_{W^{k,2m,h+2m}_{t,x_{1},x';p}(Q^{+}_{r_1})}\le C\|v\|_{L_{p}(Q_{r_2}^+)}.
\]
Observe that Theorem \ref{them:parabSobImb} implies for $\nu=1-\frac{1}{p}$,
\[
v,\,D^{2m-1}v\in C^{\frac{\nu}{2m},\nu}(Q_{r_1}^{+})
\]
and
\begin{equation}\label{eq:iterationlemmaspace}
[v]_{C^{\frac{\nu}{2m},\nu}(Q^{+}_{r_1})}+
[D^{2m-1}v]_{C^{\frac{\nu}{2m},\nu}(Q^{+}_{r_1})}\le C \|v\|_{W^{k,2m,h+2m}_{t,x_{1},x';p}(Q^{+}_{r_1})}\le C\|v\|_{L_{p}(Q_{r_2}^+)}.
\end{equation}
Since $v_{t}$ satisfies the same equation as $v$, we have
\begin{equation}\label{eq3.46}
[v_t]_{C^{\frac{\nu}{2m},\nu}(Q^{+}_{r_1})}\le C\|v_t\|_{L_{p}(Q_{r_2}^+)}.
\end{equation}
In order to show \eqref{eq:estformoe}, we need to apply the following Poincar\'e type inequality for solutions to equations satisfying the Lopatinskii--Shapiro condition. Its proof is postponed to the end of this section.
\begin{lemma}\label{lemma:poly}
Let $v\in W^{1,2m}_{p}(Q_{r_2}^{+})$ be a solution to the homogeneous problem
\eqref{prob:iterationlemma}. Then there exists a polynomial $P$ of order $2m-2$ such that
$v-P$ satisfies \eqref{prob:iterationlemma} and
there exists a constant $C=C(d,m,p,K,b_{j\beta},r_2)>0$ such that
\begin{equation}\label{eq:poincLemmaPoly}
\|D^{\alpha}(v-P)\|_{L_{p}(Q_{r_2}^{+})}\le C\|D^{2m-1}v\|_{L_{p}(Q_{r_2}^{+})}
\end{equation}
for $|\alpha|\in\{0,\ldots,2m-2\}$.
\end{lemma}
			
By \eqref{eq:iterationlemmaspace} and Lemma \ref{lemma:poly} there exists a polynomial $P$ of order $2m-2$ such that
\begin{equation*}
\begin{aligned}
&[D^{2m-1}v]_{C^{\frac{\nu}{2m},\nu}(Q_{r_1}^+)}
=[D^{2m-1}(v-P)]_{C^{\frac{\nu}{2m},\nu}(Q_{r_1}^+)}\\
&\le \|v-P\|_{L_{p}(Q_{r_2}^+)}\le C \|D^{2m-1}v\|_{L_{p}(Q_{r_2}^+)},
\end{aligned}
\end{equation*}
from which, since $D_{x'}v$ satisfies the same equation as $v$, we get that
\[
[D^{2m-1}D_{x'}v]_{C^{\frac{\nu}{2m},\nu}(Q_{r_1}^+)}\le C\|D^{2m}v\|_{L_{p}(Q_{r_2}^+)}.
\]
Together with \eqref{eq3.46}, the above inequality yields \eqref{eq:estformoe}.
\end{proof}

Similar to \cite[Corollary 5]{DK11}, from Lemma \ref{lemma:iteration} we obtain the following mean oscillation estimates for $u_{t}$ and $D^{\alpha}u$, for all $0\le|\alpha|\le 2m$ except $D_{x_1}^{2m}u$.	

\begin{lemma}\label{lemma:moe1}
Let $\kappa\geq 16$ and $p\in (1,\infty)$. Let $f\in L_{p,loc}(\overline{\R^{d+1}_{+}})$, $X_0=(t_0,x_0)\in\overline{\R^{d+1}_{+}}$, and $\lambda\geq 0$. Assume that for $r\in(0,\infty)$,  $u\in W^{1,2m}_{p,loc}(\overline{\R^{d+1}_{+}})$ satisfies $u_{t}+(A_0+\lambda)u=f$ in $Q_{\kappa r}^{+}(X_0)$ and $B_j u|_{x_1=0}=0$ on $Q_{\kappa r}(X_0)\cap\{x_{1}=0\}$, $j=1,\ldots,m$.
Assume that for some $\theta\in(0,\pi/2)$ the \ref{as:LScond}-condition is satisfied. Then
\begin{multline}\label{eq:lemmamoe1}
(|u_{t}-(u_{t})_{Q^{+}_{r}(X_0)}|)_{Q_{r}^{+}(X_0)}
+\sum_{\substack{|\alpha|\le 2m\\ \alpha_1<2m}}\lambda^{1-\frac{|\alpha|}{2m}}(|D^{\alpha}u-(D^{\alpha}u)_{Q_{r}^{+}(X_0)}|)_{Q_{r}^{+}(X_0)}\\
\le C\kappa^{-(1-\frac{1}{p})}\sum_{|\alpha|\le 2m}\lambda^{1-\frac{|\alpha|}{2m}}(|D^{\alpha}u|^{p})^{\frac{1}{p}}_{Q^{+}_{\kappa r}(X_0)}+C\kappa^{\frac{d+2m}{p}}(|f|^{p})^{\frac{1}{p}}_{Q^{+}_{\kappa r}(X_0)},
\end{multline}
where $C=C(\theta,d,m,K,p,b_{j\beta})>0$ is a constant.
\end{lemma}

\begin{proof}
Using a scaling argument, it suffices to prove \eqref{eq:lemmamoe1} only for $r=8/\kappa$.
Indeed, assume that the inequality \eqref{eq:lemmamoe1} holds true for $r=8/\kappa$. For a given $r\in (0,\infty)$, let $r_0=8/\kappa$, $R=r/r_0$ and $v(t,x)=u(R^{2m}t,Rx)$. Then $v$ satisfies $B_{j}v=0$ on $Q_{\kappa r_0}^{+}(Z_0)\cap \{x_{1}=0\}$ and
\begin{equation}\label{eq:scalinglemmamoe1}
v_t(t,x)+\sum_{|\alpha|=2m}{\color{red}\bar a_{\alpha}}
D^{\alpha}v(t,x)+\lambda R^{2m} v(t,x)= R^{2m}f(R^{2m}t,Rx)
\end{equation}
on $Q_{\kappa r_0}^{+}(Z_0)$, where $Z_0=(R^{-2m}t_0,R^{-1}x_0)\in\overline{\R^{d+1}_{+}}$. Then, by \eqref{eq:lemmamoe1} applied to \eqref{eq:scalinglemmamoe1}, we have
\begin{align*}
&(|v_{t}-(v_{t})_{Q^{+}_{r_0}(Z_0)}|)_{Q_{r_0}^{+}(Z_0)}
+\sum_{\substack{|\alpha|\le 2m\\ \alpha_1<2m}}\lambda^{1-\frac{|\alpha|}{2m}}R^{2m-|\alpha|}(|D^{\alpha}v-(D^{\alpha}v)_{Q_{r_0}^{+}(Z_0)}|)_{Q_{r_0}^{+}(Z_0)}\\
&\le C\kappa^{-(1-\frac{1}{p})}\sum_{|\alpha|\le 2m}\lambda^{1-\frac{|\alpha|}{2m}}R^{2m-|\alpha|}(|D^{\alpha}v|^{p})^{\frac{1}{p}}_{Q^{+}_{\kappa r_0}(Z_0)}+C\kappa^{\frac{d+2m}{p}}R^{2m}(|f|^{p})^{\frac{1}{p}}_{Q^{+}_{\kappa r_0}(Z_0)}.
\end{align*}
Note that
\[
(D^{\alpha}v)_{Q_{r_0}^{+}(Z_0)}=R^{|\alpha|}(D^{\alpha}u)_{Q_{r}^{+}(X_0)}\quad  {\rm and}\quad   (v_{t})_{Q_{r_0}^{+}(Z_0)}=R^{2m}(u_{t})_{Q_{r}^{+}(X_0)},\]
so the above inequality implies \eqref{eq:lemmamoe1} for arbitrary $r\in (0,\infty)$.
			
We now assume $r=8/\kappa$ and consider two cases, where we denote by $x_0^1$ the first coordinate of $x_0$.
			
\emph{Case 1: $x_0^1\geq 1$}. In this case, $Q_{\kappa r/8}^{+}(X_0)=Q_{\kappa r/8}(X_0)$. The proof of \eqref{eq:lemmamoe1} then follows from \cite[Lemma 5.7]{DK16}, with $\kappa\geq 2$ instead of $\kappa\geq 8$ there. Note that in this case, the \ref{as:LScond}-condition is not needed.
			
\emph{Case 2: $x_0^1\in[0,1]$}. We denote $Y_0:=(t_0,0,x'_{0})$ and we set $Q'_{\kappa r}(Y_0):=(t_0-(\kappa r)^{2m},t_0)\times B_{\kappa r}(x_0')$.
Observe that
\[Q_{r}^{+}(X_0)\subset Q_{2}^{+}(Y_0)\subset Q_4^{+}(Y_0)\subset Q_{6}^{+}(Y_0)\subset Q_{\kappa r}^{+}(X_0).\]
To prove \eqref{eq:lemmamoe1}, we proceed by three steps.
		
\emph{Step 1.} We assume for simplicity $Y_0=(0,0)$, since a translation in $t$ and $x'$ then gives the result for general $Y_0$. Decompose $u=v+w$ where:
\begin{itemize}
\item $w\in W^{1,2m}_{p}(\R^{d+1}_{+})$ is the solution to the inhomogeneous problem
\begin{equation}\label{prob:VMOtimespaceW}
\begin{dcases}
w_t + (A_0+\lambda)w=f\zeta & {\rm in}\ \R\times\R^{d}_{+}\\
B_j w\big|_{x_{1}=0}=0 & {\rm on}\ \partial\R^{d+1}_{+},\ j=1,\ldots,m\\
w(-6^{2m},\cdot)=0.
\end{dcases}
\end{equation}
where $\zeta\in C_0^{\infty}(\R^{d+1}_{+})$ satisfies $\zeta=1$ in $(-4^{2m},0)\times B_{4}$ and $\zeta=0$ outside $(-6^{2m},6^{2m})\times B_{6}$.
\item $v\in W^{1,2m}_{p,loc}(\R^{d+1}_{+})$ is the solution to the homogeneous problem
\begin{equation}\label{prob:VMOtimespaceV}
\begin{dcases}
v_t + (A_0+\lambda)v=0 & {\rm in}\ Q_{4}^{+}\\
B_j v\big|_{x_1=0}=0 & {\rm on}\ Q_{4}\cap\{x_1=0\},\ j=1,\ldots,m.
\end{dcases}
\end{equation}
\end{itemize}
		
\emph{Step 2.} It follows directly from Lemma \ref{lemma:estconstcoeff} with $g_j\equiv 0$ that there exists a unique solution $w\in W^{1,2m}_{p}(\R^{d+1}_+)$ of \eqref{prob:VMOtimespaceW} that satisfies
\begin{multline}\label{eq:proofmoe1}
\|w_{t}\|_{L_{p}(\R^{d+1}_{+})}+\sum_{|\alpha|\le 2m}\lambda^{1-\frac{|\alpha|}{2m}}\|D^{\alpha}w\|_{L_{p}(\R^{d+1}_{+})}\le C\|f\zeta\|_{L_{p}(\R^{d+1}_+)}\\
\le C\|f\|_{L_{p}(Q_{6}^{+})}\le C\|f\|_{L_{p}(Q_{\kappa r}^{+}(X_0))},
\end{multline}
where $C=C(\theta,K,d,m,p,b_{j\beta})$.
In particular,
\begin{equation}\label{eq:proofmoew}
(|w_{t}|^{p})_{Q_{r}^{+}}^{1/p}+\sum_{|\alpha|\le 2m}\lambda^{1-\frac{|\alpha|}{2m}}(|D^{\alpha}w|^{p})^{1/p}_{Q_{r}^{+}}\le C\kappa^{\frac{d+2m}{p}}(|f|^{p})^{1/p}_{Q_{\kappa r}^{+}}.
\end{equation}
			
\emph{Step 3.} We claim that there exists a constant $C=C(\theta,p,K,d,m,b_{j\beta})$ such that
\begin{multline}\label{eq:proofmoev}
(|v_{t}-(v_{t})_{Q_{r}^{+}}|)_{Q_r^+(X_0)}+
\sum_{\substack{|\alpha|\le 2m\\
\alpha_1<2m}}\lambda^{1-\frac{|\alpha|}{2m}}
(|D^{\alpha}v-(D^{\alpha}v)_{Q_{r}^{+}(X_0)}|)_{Q_{r}^{+}(X_0)}\\ \le C\kappa^{-\nu}\sum_{|\alpha|\le 2m}\lambda^{1-\frac{|\alpha|}{2m}}(|D^{\alpha}v|^{p})^{1/p}_{Q^{+}_{\kappa r}(X_0)}.
\end{multline}
To show the claim, we first assume $\lambda=0$. We apply Lemma \ref{lemma:iteration}
with the choice $r_1=2$ and $r_2=4$, and we get
\begin{equation}\label{eq:estformoe1}
[v_t]_{C^{\frac{\nu}{2m},\nu}(Q_{2}^+)}
+[D^{2m-1}D_{x'}v]_{C^{\frac{\nu}{2m},\nu}(Q_{2}^+)}
\le C\|v_{t}\|_{L_{p}(Q_4^+)}+C\|D^{2m}v\|_{L_{p}(Q_4^+)},
\end{equation}
where $\nu=1-\frac{1}{p}$ and $C=C(\theta,p,K,d,m,b_{j\beta})$.
		
For $\lambda>0$ we follow the proof of \cite[Lemma 3]{DK11}, based on an idea by S. Agmon. Consider for $y\in\R$,
\[
\zeta(y)=\cos(\lambda^{\frac{1}{2m}}y)+\sin(\lambda^{\frac{1}{2m}}y).
\]
Note that
\[
D_{y}^{2m}\zeta(y)=\lambda\zeta(y),\quad  \zeta(0)=1,\quad  |D^{2m-|\alpha|}\zeta(0)|=\lambda^{1-\frac{|\alpha|}{2m}}.
\]
Denote by $(t,z)=(t,x,y)\in\R^{d+2}_{+}$, where $z=(x,y)\in\R^{d+1}_{+}$ with $x\in\R^d_+$, and set
\[
\tilde{v}(t,z)=v(t,x)\zeta(y),\quad \tilde{Q}_{r}^+=(-r^{2m},0)\times\big\{|z|<r,z\in\R^{d+1}_{+}\big\}.
\]
Since $v$ satisfies \eqref{prob:VMOtimespaceV} on $Q_4^+$, $\tilde{v}$ satisfies
\begin{equation*}
\begin{dcases}
\tilde{v}_{t}+ A_0\tilde{v} + D^{2m}_{y}\tilde{v}=0 & {\rm in}\ \tilde{Q}_{4}^{+}\\
B_{j}\tilde v|_{x_1=0}=0 & {\rm on}\ \tilde{Q}_{4}^{+}\cap \{x_1=0\}.
\end{dcases}
\end{equation*}
Thus, we can proceed as in \eqref{eq:estformoe1} and get for $r=8/\kappa$, $ \kappa\geq 16$, and $|\alpha|\le 2m$ with $\alpha_{1}<2m$,
\begin{equation}\label{eq:moefortildev}
[\tilde{v}_{t}]_{C^{\frac{\nu}{2m},\nu}(\tilde{Q}_{2}^+)}+[D^{2m-|\alpha|}_{y}D^{\alpha}\tilde{v}]_{C^{\frac{\nu}{2m},\nu}(\tilde{Q}_{2}^+)}
\le C\|\tilde{v}_{t}\|_{L_{p}(\tilde{Q}_4^+)}+C\|D^{2m}\tilde{v}\|_{L_{p}(\tilde{Q}_4^+)}.
\end{equation}
Since   $|D^{2m-|\alpha|}\zeta(0)|=\lambda^{1-\frac{|\alpha|}{2m}}$,
\begin{equation*}
\lambda^{1-\frac{|\alpha|}{2m}}[D^{\alpha}v]_{C^{\frac{\nu}{2m},\nu}(Q_{2}^{+})}\le [D^{2m-|\alpha|}_{y}D^{\alpha}\tilde{v}]_{C^{\frac{\nu}{2m},\nu}(\tilde{Q}_{2}^+)}.
\end{equation*}
Observe now that
\begin{equation*}
\begin{aligned}
(|D^{\alpha} v-(D^{\alpha}v)_{Q_{r}^{+}(X_0)}|)_{Q_{r}^{+}(X_0)}&\le C \kappa^{-\nu}[D^{\alpha}v]_{C^{\frac{\nu}{2m},\nu}(Q_{r}^{+}(X_0))}\\
&\le C \kappa^{-\nu}[D^{\alpha}v]_{C^{\frac{\nu}{2m},\nu}(Q_{2}^{+})}
\end{aligned}
\end{equation*}
and the same holds for $v_t$. This implies that
\begin{equation*}
\begin{aligned}
&(|v_{t}-(v_{t})_{Q_{r}^{+}(X_0)}|)_{Q_{r}^{+}(X_0)}+\lambda^{1-\frac{|\alpha|}{2m}}(|D^{\alpha}v-(D^{\alpha}v)_{Q_{r}^{+}(X_0)}|)_{Q_{r}^{+}(X_0)}\\
&\le C\kappa^{-\nu}[v_{t}]_{C^{\frac{\nu}{2m},\nu}(Q_{2}^{+})}+C\kappa^{-\nu}\lambda^{1-\frac{|\alpha|}{2m}}[D^{\alpha}v]_{C^{\frac{\nu}{2m},\nu}(Q_{2}^{+})}\\
&\le C\kappa^{-\nu}[\tilde{v}_{t}]_{C^{\frac{\nu}{2m},\nu}(\tilde{Q}_{2}^+)}+ C\kappa^{-\nu}[D^{2m-|\alpha|}_{y}D^{\alpha}\tilde{v}]_{C^{\frac{\nu}{2m},\nu}(\tilde{Q}_{2}^{+})}.
\end{aligned}
\end{equation*}
Therefore, the left-hand side of \eqref{eq:proofmoev} is bounded by that of \eqref{eq:moefortildev}.

Since $D^{2m}\tilde{v}$ is a linear combination of terms such as
\[
\lambda^{ 1-\frac{k}{2m}}\cos(\lambda^{\frac{1}{2m}}y)D^{k}_{x}u(t,x),\quad \lambda^{ 1-\frac{k}{2m}}\sin(\lambda^{\frac{1}{2m}}y)D^{k}_{x}u(t,x),\ \ k=0,\ldots,2m,
\]
we have
$$
\|D^{2m}\tilde{v}\|_{L_{p}(\tilde{Q}_{4}^{+})}\le C\sum_{|\alpha|\le 2m}\lambda^{1-\frac{|\alpha|}{2m}} \|D^{\alpha}v\|_{L_{p}(Q_{\kappa r}^{+}(X_0))}.
$$
This together with $v_t = -A_0v$ yields
\[
C\kappa^{-\nu}\|\tilde{v}_{t}\|_{L_{p}(\tilde{Q}_4^+)}+C\kappa^{-\nu}\|D^{2m}\tilde{v}\|_{L_{p}(\tilde{Q}_4^+)}\le C\kappa^{-\nu} \sum_{|\alpha|\le 2m}\lambda^{1-\frac{|\alpha|}{2m}} \|D^{\alpha}v\|_{L_{p}(Q_{\kappa r}^{+}(X_0))},
\]
which shows that the right-hand side of \eqref{eq:moefortildev} is bounded by that of \eqref{eq:proofmoev}.
			
\emph{Step 4.}
Since $u=w+v$, by \eqref{eq:proofmoew} and \eqref{eq:proofmoev} we get
\begin{equation*}
\begin{aligned}
&(|u_{t}-(u_{t})_{Q_{r}}|)_{Q_r^+(X_0)}+\sum_{|\alpha|\le 2m,\alpha_1<2m}\lambda^{1-\frac{|\alpha|}{2m}}
(|D^{\alpha}u-(D^{\alpha}u)_{Q_{r}^{+}(X_0)}|)_{Q_{r}^{+}(X_0)}\\
&\stackrel{(i)}{\le} C(|u_{t}-(v_{t})_{Q_{r}}|)_{Q_r^+(X_0)}+C\sum_{|\alpha|\le 2m,\alpha_1<2m}\lambda^{1-\frac{|\alpha|}{2m}}(|D^{\alpha} u-(D^{\alpha}v)_{Q_{r}^{+}(X_0)}|)_{Q_{r}^{+}(X_0)}\\
&\le C(|v_{t}-(v_{t})_{Q_{r}}|)_{Q_r^+(X_0)}+C\sum_{|\alpha|\le 2m,\alpha_1<2m}\lambda^{1-\frac{|\alpha|}{2m}}(|D^{\alpha} v-(D^{\alpha}v)_{Q_{r}^{+}(X_0)}|)_{Q_{r}^{+}(X_0)}\\
&\quad +C(|w_{t}|^{p})^{1/p}_{Q_r^+(X_0)}+C\sum_{|\alpha|\le 2m,\alpha_1<2m}\lambda^{1-\frac{|\alpha|}{2m}}(|D^{\alpha}w|^{p})^{1/p}_{Q_{r}^{+}(X_0)}\\
&\le C \kappa^{-\nu}\sum_{|\alpha|\le 2m,\alpha_1<2m}\lambda^{1-\frac{|\alpha|}{2m}}(|D^{\alpha}v|^{p})^{1/p}_{Q_{\kappa r}^{+}(X_0)}+C \kappa^{\frac{d+2m}{p}}(|f|^{p})^{1/p}_{Q_{\kappa r}^{+}(X_0)}\\
&\stackrel{(ii)}{\le}C \kappa^{-\nu}\sum_{|\alpha|\le 2m,\alpha_1<2m}\lambda^{1-\frac{|\alpha|}{2m}}
(|D^{\alpha}u|^{p})^{1/p}_{Q_{\kappa r}^{+}(X_0)}+C \kappa^{\frac{d+2m}{p}}(|f|^{p})^{1/p}_{Q_{\kappa r}^{+}(X_0)},
\end{aligned}
\end{equation*}
where in (i) we used the fact that for any constant $c_1,c_2$ it holds
\begin{align*}
(|u_t-(u_t)_{Q_{r}^{+}(X_0)}|)_{Q_{r}^{+}(X_0)}&\le
2(|u_t - c_1|)_{Q_{r}^{+}(X_0)},
\\
(|D^{\alpha} u-(D^{\alpha}u)_{Q_{r}^{+}(X_0)}|)_{Q_{r}^{+}(X_0)}&\le
2(|D^{\alpha} u-c_2|)_{Q_{r}^{+}(X_0)},
\end{align*}
 and we took $c_1=(v_t)_{Q_{r}^{+}(X_0)}$, $c_2=(D^{\alpha}v)_{Q_{r}^{+}(X_0)}$, while in (ii) we used $v=u-w$ and \eqref{eq:proofmoe1}.
\end{proof}
	
We now use the idea of freezing the coefficients as in \cite[Lemma 5.9]{DK16}, to obtain the following mean oscillation estimate on $Q_{r}^{+}(X_0)$ for operators with variable coefficients when $r$ is small.
\begin{lemma}\label{lemma:moe2}
Let $\lambda\geq 0$ and $\kappa\geq 16$. Assume that $A$ and $B_j$, $j=1,\ldots, m$, satisfy conditions \ref{as:operatorALS}, \ref{as:operatorBLS}, and \ref{as:LScond} for some $\theta\in(0,\pi/2)$, and assume the lower-order coefficients of $A$ and $B_j$ to be all zero. Let $\mu,\varsigma\in(1,\infty)$, $\frac{1}{\varsigma}+\frac{1}{\mu}=1$.
Then, for $r\in(0,R_0/\kappa]$, $X_0\in\overline{\R^{d+1}_{+}}$ and $u\in W^{1,2m}_{p\mu,loc}(\overline{\R^{d+1}_{+}})$ satisfying $u_{t}+(A(t)+\lambda)u=f$ in $Q_{\kappa r}^{+}(X_0)$ and $B_j u|_{x_1=0}=0$ on $Q_{\kappa r}(X_0)\cap\{x_{1}=0\}$, $j=1,\ldots,m$, where $f\in L_{p,loc}(\overline{\R^{d+1}_{+}})$,  we have
\begin{multline*}
(|u_{t}-(u_{t})_{Q^{+}_{r}(X_0)}|)_{Q^{+}_{r}(X_0)}
+\sum_{\substack{|\alpha|\le 2m\\ \alpha_1<2m}}\lambda^{1-\frac{|\alpha|}{2m}}(|D^{\alpha}u-(D^{\alpha}u)_{Q^{+}_{r}(X_0)}|)_{Q^{+}_{r}(X_0)}\\
\le C\kappa^{-(1-\frac{1}{p})}\sum_{|\alpha|\le 2m}\lambda^{1-\frac{|\alpha|}{2m}}(|D^{\alpha}u|^{p})^{\frac{1}{p}}_{Q^{+}_{\kappa r}(X_0)}+C\kappa^{\frac{d+2m}{p}}(|f|^{p})^{\frac{1}{p}}_{Q^{+}_{\kappa r}(X_0)}\\
+C\kappa^{\frac{d+2m}{p}}\rho^{\frac{1}{p\varsigma}}(|D^{2m}u|^{p\mu})^{\frac{1}{p\mu}}_{Q^{+}_{\kappa r}(X_0)},
\end{multline*}
where $C=C(\theta,d,m,\mu,K,p,b_{j\beta})>0$.
\end{lemma}
		
\begin{proof}
Fix $(t_0,x_0)\in\overline{\R^{d+1}_{+}}$. For any $(s,y)\in Q_{\kappa r}^{+}(t_0,x_0)$, set
\[
A_{s,y}u=\sum_{|\alpha|=2m}a_{\alpha}(s,y)D^{\alpha}u.
\]
Then $u$ satisfies
\begin{equation*}
\begin{dcases}
u_t+(A_{s,y}+\lambda)u=g & {\rm in}\ Q_{\kappa r}^{+}\\
B_j u\big|_{x_1=0}=0 & {\rm on}\ Q_{\kappa r}^{+}\cap \{x_1=0\},
\end{dcases}
\end{equation*}
where
\[
g:=f+\sum_{|\alpha|=2m}(a_{\alpha}(s,y)-a_{\alpha}(t,x))D^{\alpha}u.
\]
Note that when $x_0^1\le R_0$, we have $y_1\le 2R_0$ so that the \ref{as:LScond}-condition is satisfied for $A_{s,y}$ and $B_j$.
It follows from Lemma \ref{lemma:moe1} that
\begin{multline}\label{eq:1proofmoe2}
(|u_{t}-(u_{t})_{Q^{+}_{r}(X_0)}|)_{Q^{+}_{r}(X_0)}
+\sum_{|\alpha|\le 2m,\alpha_1<2m}\lambda^{1-\frac{|\alpha|}{2m}}(|D^{\alpha}u-(D^{\alpha}u)_{Q_{r}^{+}(X_0)}|)_{Q_{r}^{+}(X_0)}\\
\le C\kappa^{-(1-\frac{1}{p})}\sum_{|\alpha|\le 2m}\lambda^{1-\frac{|\alpha|}{2m}}(|D^{\alpha}u|^{p})^{\frac{1}{p}}_{Q^{+}_{\kappa r}(X_0)}+C\kappa^{\frac{d+2m}{p}}(|g|^{p})^{\frac{1}{p}}_{Q^{+}_{\kappa r}(X_0)},
\end{multline}
where $C=C(\theta,d,m,K,p,b_{j\beta})$. Note that
\begin{equation}\label{eq:2proofmoe2}
(|g|^{p})^{\frac{1}{p}}_{Q^{+}_{\kappa r}(X_0)}\le (|f|^{p})^{\frac{1}{p}}_{Q^{+}_{\kappa r}(X_0)}+I^{\frac{1}{p}},
\end{equation}
where
\[
I=(|(a_{\alpha}(s,y)-a_{\alpha}(t,x))D^{\alpha}u|^{p})_{Q^{+}_{\kappa r}(X_0)}.
\]
Take now the average of $I$ with respect to $(s,y)$ in $Q^{+}_{\kappa r}(X_0)$.
By H\"{o}lder's inequality it holds that
\begin{equation*}
\begin{aligned}
&\Big( \dashint_{Q^{+}_{\kappa r}(X_0)}I\, ds\,dy \Big)^{\frac{1}{p}}\le\Big(\dashint_{Q^{+}_{\kappa r}(X_0)}(|(a_{\alpha}(s,y)-a_{\alpha}(t,x))D^{\alpha}u|^{p})_{Q^{+}_{\kappa r}(X_0)}\, ds\, dy \Big)^{\frac{1}{p}}\\
&\le \Big(\dashint_{Q^{+}_{\kappa r}(X_0)}(|(a_{\alpha}(s,y)-a_{\alpha}(t,x))|^{p\varsigma})^{\frac{1}{\varsigma}}_{Q^{+}_{\kappa r}(X_0)} \,ds\, dy \Big)^{\frac{1}{p}}(|D^{2m}u|^{p\mu})_{Q^{+}_{\kappa r}(X_0)}^{\frac{1}{p\mu}}.
\end{aligned}
\end{equation*}
Moreover, by the boundedness of the coefficients $a_{\alpha}$, the assumption $r\le R_0/\kappa$ and Assumption \ref{ass:VMO} $(\rho)$, we get
\begin{equation*}
\begin{aligned}
&\Big(\dashint_{Q^{+}_{\kappa r}(X_0)}(|(a_{\alpha}(s,y)-a_{\alpha}(t,x)|^{p\varsigma})_{Q^{+}_{\kappa r}(X_0)})^{\frac{1}{\varsigma}}\Big)^{\frac{1}{p}}\\
&\le\Big(\dashint_{Q^{+}_{\kappa r}(X_0)}(|a_{\alpha}(s,y)-a_{\alpha}(t,x)|)_{Q^{+}_{\kappa r}(X_0)}\,ds\,dy\Big)^{\frac{1}{p\varsigma}}\\
&\le C(\osc(a_{\alpha},Q_{\kappa r}^{+}))^{\frac{1}{p\varsigma}}\le C((a_{\alpha})^{\sharp}_{R_0})^{\frac{1}{p\varsigma}}\le  C\rho^{\frac{1}{p\varsigma}}.
\end{aligned}
\end{equation*}		
This together with \eqref{eq:1proofmoe2} and \eqref{eq:2proofmoe2} gives the desired estimate. When $x_0^1>R_0$, the results follows directly by \cite[Lemma 5]{DK11}, since in this case there are no boundary conditions involved. \end{proof}

We conclude this section with the proof of Lemma \ref{lemma:poly}.
		
\begin{proof}[Proof of Lemma \ref{lemma:poly}] Without loss of generality we can take $r_2=1$. We take for simplicity the center $X_0$ of $Q_{1}^{+}$ to be $(0,0)$. A translation of the coordinates then gives the result for general $X_0\in \partial\R^{d+1}_{+}$.
			
Assume that the polynomial $P$ has the form
\[
P=\sum_{|\alpha|\le 2m-2}{\frac{c_{\alpha}}{\alpha !}}x^{\alpha},\quad x=(x_1,x')\in\R^{n}_{+},\quad \alpha !=\alpha_{1} !\cdots\alpha_{d}!
\]
and satisfies the boundary conditions
\begin{equation}\label{eq:boundpol}
B_{j}P\big|_{x_1=0}=\sum_{|\beta|= m_j}b_{j\beta}D^{\beta}P\Big|_{x_1=0}=0,
\end{equation}
where $j=1,\ldots,m$ and $0\le m_j\le 2m-1$. Since $P$ is of order $2m-2$, we only need to consider the boundary conditions whose order is $m_{j}\le 2m-2$.

Assume that the \ref{as:LScond}-condition is satisfied. Then, the boundary operators $B_1,\ldots,B_m$ are linearly independent, and so are their tangential derivatives $ D^{\gamma}_{x'}B_{j}$.
			
To determine the coefficients $c_{\alpha}$ of the polynomial, we proceed by induction on the value of $|\alpha|$.  For this, we introduce two subgroups of multi-indices:
\begin{align*}
\mathcal{I}_{|\alpha|}&:=\big\{\alpha\in\N_0^d :\ c_{\alpha}\ {\rm are\ determined\ using\ the\ boundary\ conditions}\big\}\\
\mathcal{J}_{|\alpha|}&:=\big\{\alpha\in\N_0^d :\ c_{\alpha}\ {\rm are\ determined\ using\ the\ condition}\ (D^{\alpha}P)_{Q_{1}^{+}}=(D^{\alpha}v)_{Q_1^+}\big\}.
\end{align*}
		
\emph{Step 1.} Let $|\alpha|=2m-2$ and $m_{j}\le 2m-2$. We will first determine the coefficients $c_\alpha$ and then prove the Poincar\'e type inequality
\begin{equation}\label{eq:Poinc2m-2}
\|D^{\alpha}(v-P)\|_{L_{p}(Q_1^+)}\le C\|D^{2m-1}v\|_{L_{p}(Q_1^+)}.
\end{equation}

For this, we take the $2m-2-m_{j}$-th tangential derivatives of each boundary condition in \eqref{eq:boundpol} and setting $x'=0$ we get a system of equations of the form	\begin{equation}\label{eq:lemmapoly1}
\sum_{|\beta|=m_{j}}b_{j\beta}c_{\beta+\gamma}=0,
\end{equation}
each $\gamma$ satisfying $|\gamma|=2m-2-m_{j}$, so that $|\beta+\gamma|=2m-2$, and $\gamma_{1}=0$.
			
We rewrite the above system as the product of the $r\times n$ matrix $\displaystyle \mathbb{B}=[b_{j\beta}^{i,\ell}]_{i=1,\ell=1}^{r,n}$ of the coefficients $b_{j\beta}$ by the vector $\textbf{C}=(c_{\alpha}^{\ell}:\ |\alpha|=2m-2)_{\ell=1}^{n}$ of the coefficients $c_\alpha$, where $n$ denotes the number of the unknown $c_\alpha$'s and $r$ the number of the equations in \eqref{eq:lemmapoly1}.
			
By the \ref{as:LScond}-condition, the $r$ rows of $\mathbb{B}$ are linearly independent. This implies that there exists an $r\times r$ submatrix $\mathbb{B}_{1}$ of $\mathbb{B}$ such that $rank(\mathbb{B}_1)=r$. Define $\mathbb{B}_2:=\mathbb{B}-\mathbb{B}_1$. Consider the vectors $\textbf{C} _{1}:=(c_\alpha^{i}:\ \alpha\in\mathcal{I}_{2m-2})_{i=1}^{r}$ and
$\textbf{C}_{2}:=(c_\alpha^{k}:\ \alpha\in\mathcal{J}_{2m-2})_{k=1}^{n-r}$.
We then rewrite the equation $\mathbb{B}\textbf{C} =0$ as $\mathbb{B}_1\textbf{C} _{1}=-\mathbb{B}_2\textbf{C} _{2}$, and we get
\[
\textbf{C} _{1}=-\mathbb{B}_1^{-1}\mathbb{B}_2\textbf{C} _{2}.
\]
From this we obtain that the coefficients $c_\alpha$ with $\alpha\in\mathcal{I}_{2m-2}$ depends on the coefficients $c_\alpha$ with $\alpha\in\mathcal{J}_{2m-2}$.
			
We determine the last ones by requiring
\begin{equation*}
(D^{\alpha}P)_{Q_{1}^{+}}=(D^{\alpha}v)_{Q_{1}^{+}},\quad \alpha\in \mathcal{J}_{2m-2}.
\end{equation*}
We then apply the interior Poincar\'e inequality as in \cite[Lemma 3.3]{DKBMO11} and we get
\begin{equation}\label{eq:Poinc2m-2p1}
\begin{aligned}
\|D^{\alpha}(v-P)\|_{L_{p}(Q_1^+)}&\le C_0\|D^{2m-1}(v-P)\|_{L_{p}(Q_1^+)}\\
&=C_0\|D^{2m-1}v\|_{L_{p}(Q_1^+)},
\end{aligned}
\end{equation}
with $\alpha\in\mathcal{J}_{2m-2}$ and $C_0=C_0(d,m,p)$.

Now let $\mathbf{D}^{\alpha}(v-P)$ be the vector of the derivatives $D^{\alpha}(v-P)$ for any multi-index $\alpha$, $\mathbf{B}(v-P)$ be the vector with components $B_j(v-P)$, and $\mathbf{D}^{\gamma}_{x'}\mathbf{B}(v-P)$ be the vector with components $D^{\gamma}_{x'}B_{j}(v-P)$. Observe that
\begin{equation}\label{eq:proofpoly1}
\mathbb{B} \mathbf{D}^{\alpha}(v-P) = \mathbf{D}^{\gamma}_{x'}\mathbf{B}(v-P),
\end{equation}
where $|\gamma|+m_{j}=|\alpha|=2m-2$.

Furthermore, let $\mathbf{D^{\alpha}_{\mathcal{I}}}(v-P)$ and $\displaystyle \mathbf{D^{\alpha}_{\mathcal{J}}}(v-P)$ denote the vectors with components $D^{\alpha}(v-P)$ with respectively $\alpha\in \mathcal{I}_{2m-2}$ and $\alpha\in\mathcal{J}_{2m-2}$. Observe that the order of their components depends respectively on the order of the components in the vectors $\mathbf{C}_1$ and $\mathbf{C}_2$ defined above.
Thus, for $\mathbb{B}_1$ and $\mathbb{B}_2$ introduced above, it holds that
\[
\mathbb{B} \mathbf{D}^{\alpha}(v-P)= \mathbb{B}_1 \mathbf{D^{\alpha}_{\mathcal{I}}}(v-P) + \mathbb{B}_2 \mathbf{D^{\alpha}_{\mathcal{J}}}(v-P).
\]

This, combined with \eqref{eq:proofpoly1}, implies that
\begin{equation}\label{eq:matrix1}
\mathbb{B}_1 \mathbf{D}_{\mathcal{I}}^{\alpha}(v-P)= \mathbf{D}^{\gamma}_{x'}\mathbf{B}(v-P) - \mathbb{B}_2 \mathbf{D^{\alpha}_{\mathcal{J}}}(v-P).
\end{equation}

Since $D^{\gamma}_{x'}B_j(v-P)=0$ on the boundary, we can apply the boundary Poincar\'e inequality and  we get
\begin{equation}\label{eq:Poinc2m-2p2}
\|D^{\gamma}_{x'}B_j(v-P)\|_{L_{p}(Q_1^+)}\le C_1\|D^{2m-1}v\|_{L_{p}(Q_1^+)},\quad  C_1=C_1(d,m,p,K).
\end{equation}
By \eqref{eq:matrix1} and combining \eqref{eq:Poinc2m-2p1} and \eqref{eq:Poinc2m-2p2}, we get
\[
\|D^{\alpha}(v-P)\|_{L_p(Q_1^+)}\le (\det(\mathbb{B}_{1}))^{-1} C_2\|D^{2m-1}v\|_{L_p(Q_1^+)},\quad \alpha\in\mathcal{I}_{2m-2},
\]
where $C_2=C_2(d,m,p,K)$. Since $\mathbb{B}_{1}$ has dimension $r\times r$ and $rank(\mathbb{B}_{1})=r$, $\det(\mathbb{B}_{1})\neq 0$. Thus, there exists $\delta>0$ small enough and depending on $b_{j\beta}$, such that $\det(\mathbb{B}_{1})>\delta$. Therefore, we obtain \eqref{eq:Poinc2m-2}, i.e.,
\[
\|D^{\alpha}(v-P)\|_{L_{p}(Q_1^+)}\le C\|D^{2m-1}v\|_{L_{p}(Q_1^+)},\quad |\alpha|=2m-2,
\]
with C depending only on $d,m,p,K$ and $b_{j\beta}$.
			
\medskip
			
\emph{Step 2.} Let $|\alpha|=2m-3$ and $m_{j}\le 2m-3$.  By taking the $(2m-3-m_{j})$-th tangential derivatives of each boundary condition in \eqref{eq:boundpol} and setting $x'=0$ we get a system of equation of the form
\[
\sum_{|\beta|=m_{j}}b_{j\beta}c_{\beta+\gamma}=0,
\]
each $\gamma$ satisfying $|\gamma|=2m-3-m_{j}$, so that $|\beta+\gamma|=2m-3$,  and $\gamma_{1}=0$.
As before, we determine the coefficients $c_\alpha$ with $\alpha\in\mathcal{I}_{2m-3}$ in terms of the coefficients $c_\alpha$ with $\alpha\in\mathcal{J}_{2m-3}$.
The last one are determined as in the previous step by requiring
\[
(D^{\alpha}P)_{Q_{1}^{+}}=(D^{\alpha}v)_{Q_{1}^{+}},\quad \alpha\in \mathcal{J}_{2m-3}.
\]
Observe that in the average condition there are coefficients $c_\alpha$ with $|\alpha|=2m-2$, but they have been already determined in {\em Step 1}. From this, proceeding as in {\em Step 1} and applying the PoincarPoincar\'e type inequality \eqref{eq:Poinc2m-2} we get
\[
\|D^{\alpha}(v-P)\|_{L_{p}(Q_{1}^{+})}\le C\|D^{2m-2}(v-P)\|_{L_{p}(Q_{1}^{+})}\le C\|D^{2m-1}v\|_{L_{p}(Q_{1}^{+})},
\]
with $|\alpha|=2m-3$ and $C$ depending only on $d,m,p,K$ and $b_{j\beta}$.
			
\medskip
		
\emph{Step k.} Let $|\alpha|=2m-1-k$ and $m_{j}\le 2m-1-k$. We proceed by induction.
			
By taking the $(2m-1-k-m_{j})$-th tangential derivatives of each boundary condition in \eqref{eq:boundpol} and setting $x'=0$ we get a system of equation of the form
\[
\sum_{|\beta|=m_j}b_{j\beta}c_{\beta+\gamma}=0,
\]
each $\gamma$ satisfying $|\gamma|=2m-1-k-m_{j}$, so that $|\beta+\gamma|=2m-1-k$,  and $\gamma_{1}=0$.
Proceeding as before, we determine the coefficients $c_\alpha$, $\alpha\in\mathcal{I}_{2m-1-k}$ in terms of the coefficients $c_{\alpha}$, $\alpha\in\mathcal{J}_{2m-1-k}$. The last ones are determined by requiring
\[
(D^{\alpha}P)_{Q_1^+}=(D^{\alpha}v)_{Q_1^+},\quad \alpha\in \mathcal{J}_{2m-1-k}.
\]
Observe that by induction we have determined the coefficients $c_{\alpha}$, $|\alpha|\in\{2m-2,\ldots,2m-k\}$.
Therefore, proceeding as in {\em Step 1}, using induction for $|\alpha|\in\{2m-2,\ldots,2m-k\}$ and applying the Poincar\'e type inequalities obtained at any induction step, we get
\[
\|D^{\alpha}(v-P)\|_{L_{p}(Q_{1}^{+})}\le C\|D^{2m-k}(v-P)\|_{L_{p}(Q_{1}^{+})}
\le\cdots\le C\|D^{2m-1}v\|_{L_{p}(Q_{1}^{+})},
\]
with $|\alpha|=2m-1-k$ and $C$ depending only on $d,m,p,K$ and $b_{j\beta}$.
		
\medskip
		
\emph{Step 2m-1.} Let $|\alpha|=0$. If $P(x)|_{x_1=0}=0$ is a boundary condition, then $c_0=0$. Otherwise, we determine $c_0$ by using the average condition $(P)_{Q_1^+}=(v)_{Q_1^+}$.

This concludes the construction of the required polynomial $P$.
Moreover, by induction we get \eqref{eq:poincLemmaPoly}.

To conclude the proof, observe that the polynomial $P$ satisfies the boundary conditions. In fact, by the construction above, at each step one can show by induction that the tangential derivatives of the boundary conditions are equal to zero.
Since the boundary conditions are satisfied at the origin $x'=0$, they must then be satisfied for any $x'\in\R^{d-1}$. The assertion follows.
\end{proof}
		
\section{$L_{p}(L_q)$-estimates for systems with general boundary condition}\label{sec:proofmainresult}
	
We are now ready to prove Theorem \ref{thm:VMOproblemLS}. For this, we will follow the procedure of \cite[Theorem 5.4]{DK16} and we will need two intermediate results. The first one follows from Lemma \ref{lemma:moe2}.
		
\begin{lemma}\label{lemma:normspq}
Let $p,q\in(1,\infty)$, $v\in A_{p}(\R)$, $w\in A_{q}(\R^d_+)$, $\lambda\geq 0$ and $t_1\in\R$. Assume that $A$ and $B_{j}, j=1,\ldots,m$, satisfy conditions \ref{as:operatorALS}, \ref{as:operatorBLS}, and \ref{as:LScond} for some $\theta\in(0,\pi/2)$, and assume the lower-order coefficients of $A$ and $B_j$ to be all zero. Then, there exists constants $R_1,\rho\in (0,1)$, depending only on $\theta$, $m$, $d$, $K$, $p$, $q$, $[v]_p$, $[w]_q$, and $b_{j\beta}$, such that for $u\in W^{1,2m}_{p,q,v,w}(\R^{d+1}_{+})$ vanishing outside $(t_1-(R_0 R_1)^{2m},t_1)\times \R^{d}_{+}$ and satisfying \eqref{prob:VMOtimespaceLS} in $\R^{d+1}_{+}$, where $f\in L_{p,q,v,w}(\R^{d+1}_{+})$, it holds that
\begin{equation}\label{eq:lemmanormspq}
\sum_{|\alpha|\le 2m}\lambda^{1-\frac{|\alpha|}{2m}}\|D^{\alpha}u\|_{L_{p,q,v,w}(\R^{d+1}_{+})}\le C\|f\|_{L_{p,q,v,w}(\R^{d+1}_{+})},
\end{equation}
where $C=C(\theta,d,m,K,p,q,[v]_{p},[w]_{q}, b_{j\beta})>0$.
\end{lemma}
		
\begin{proof}
For the given $v\in A_{p}(\R)$ and $w\in A_q(\R^{d}_{+})$, using reverse H\"{o}lder's inequality (see \cite[Corollary 9.2.4 and Remark 9.2.3]{GrafakosModern}) we find $\sigma_1=\sigma_1(p,[v]_{p})$, $\sigma_2=\sigma_2(q,[w]_{q})$ such that $p-\sigma_1>1$, $q-\sigma_2>1$ and
\[
v\in A_{p-\sigma_1}(\R),\quad  w\in A_{q-\sigma_2}(\R^{d}_{+}).
\]
Take $p_0,\mu\in (1,\infty)$ satisfying $\displaystyle p_0 \mu=\min\Big\{\frac{p}{p-\sigma_1},\frac{q}{q-\sigma_2}\Big\}>1$.
Note that
\[
v\in A_{p-\sigma_1}\subset A_{p/(p_0 \mu)}\subset A_{p/p_0}(\R),
\]
\[
w\in A_{q-\sigma_2}\subset A_{q/(p_0 \mu)}\subset A_{q/p_0}(\R^{d}_{+}).
\]
Then it holds that
\[
u\in W^{1,2m}_{p_0\mu,{\rm loc}}(\R^{d+1}_{+}),\quad  f\in L_{p_0\mu,{\rm loc}}(\R^{d+1}_{+}).
\]
Indeed, by \cite[Lemma 3.1]{DK16}, for any $g\in L_{p_0\mu,loc}$ and for any half-ball $B_1^{+}\subset\R^{d}_{+}$ and interval $B_2\subset\R$,
\begin{equation*}
\begin{aligned}
&\frac{1}{|B_1^+||B_2|}\int_{B_1^+\times B_2}|g|^{p_{0}\mu}\,dx\,dt=\frac{1}{|B_{2}|}\int_{B_2} \frac{1}{|B_1^+|}\int_{B_1^+} |g|^{p_{0}\mu}\,dx\,dt\\
&\le \frac{1}{|B_{2}|} \int_{B_2}\bigg(\frac{[w]_{q/(p_0 \mu)}}{w(B_1^+)}\int_{B_1^+}|g|^{q}w(x)\,dx\bigg)^{\frac{p_0 \mu}{q}}\,dt\\
&\le \bigg( \frac{[v]_{p/(p_0 \mu)}}{v(B_2)}\int_{B_2}\bigg(\frac{[w]_{q/(p_0 \mu)}}{w(B_1^+)}\int_{B_1^+}|g|^{q}w(x)\,dx\bigg)^{\frac{p}{q}}v(t)\,dt\bigg)^{\frac{p_0 \mu}{p}}.
\end{aligned}
\end{equation*}
Let $\kappa\geq 16$ be a large constant to be specified.
If $r>\frac{R_0}{\kappa}$, since $u$ vanishes outside $(t_1-(R_0 R_1)^{2m},t_1)\times \R^{d}_{+}$, for $0\le|\alpha|\le 2m$, we have
\begin{equation}\label{eq:lemmanormspq1}
\begin{aligned}
&(|D^{\alpha}u-(D^{\alpha}u)_{Q^{+}_{r}(X_0)}|)_{Q^{+}_{r}(X_0)}\le 2( |D^{\alpha}u|)_{Q^{+}_{r}(X_0)}\\
&\le 2 (I_{(t_1-(R_0 R_1)^{2m},t_1)}(s))_{Q^{+}_{r}(X_0)}^{1-\frac{1}{p_0}}(|D^{\alpha}u|^{p_0})_{Q^{+}_{r}(X_0)}^{\frac{1}{p_0}}\\
&\le C_{d,m,p_{0}}\kappa^{2m(1-\frac{1}{p_0})}R_{1}^{2m(1-\frac{1}{p_0})}(|D^{\alpha}u|^{p_0})_{Q^{+}_{r}(X_0)}^{\frac{1}{p_0}},
\end{aligned}
\end{equation}
where $I$ denotes the indicator function.

If $r\in(0,R_0/\kappa]$, then by Lemma \ref{lemma:moe2} with $p=p_0$, there exists a constant \linebreak $C=C(\theta,d,m,\mu,K,p_0,b_{j\beta})$ such that, for $\frac{1}{\mu}+\frac{1}{\varsigma}=1$,
\begin{equation}\label{eq:lemmanormspq2}
\begin{aligned}
&(|u_{t}-(u_{t})_{Q^{+}_{r}(X_0)}|)_{Q^{+}_{r}(X_0)}
+\sum_{\substack{|\alpha|\le 2m, \alpha_1<2m}}\lambda^{1-\frac{|\alpha|}{2m}}(|D^{\alpha}u-(D^{\alpha}u)_{Q^{+}_{r}(X_0)}|)_{Q^{+}_{r}(X_0)}\\
&\le C\kappa^{-(1-\frac{1}{p_0})}\sum_{|\alpha|\le 2m}\lambda^{1-\frac{|\alpha|}{2m}}(|D^{\alpha}u|^{p_0})^{\frac{1}{p_0}}_{Q^{+}_{\kappa r}(X_0)}+C\kappa^{\frac{d+2m}{p_0}}(|f|^{p_0})^{\frac{1}{p_0}}_{Q^{+}_{\kappa r}(X_0)}\\
&\quad +C\kappa^{\frac{d+2m}{p_0}}\rho^{\frac{1}{p_0 \varsigma}}(|D^{2m}u|^{p_0\mu})^{\frac{1}{p_0\mu}}_{Q^{+}_{\kappa r}(X_0)}.
\end{aligned}
\end{equation}
Combining \eqref{eq:lemmanormspq1} and \eqref{eq:lemmanormspq2} we get
\begin{equation*}
\begin{aligned}
&(|u_{t}-(u_{t})_{Q^{+}_{r}(X_0)}|)_{Q^{+}_{r}(X_0)}+\sum_{\substack{|\alpha|\le 2m, \alpha_1<2m}}\lambda^{1-\frac{|\alpha|}{2m}}(|D^{\alpha}u-(D^{\alpha}u)_{Q^{+}_{r}(X_0)}|)_{Q^{+}_{r}(X_0)}\\
&\le C(\kappa^{2m(1-\frac{1}{p_0})}R_{1}^{2m(1-\frac{1}{p_0})}+\kappa^{-(1-\frac{1}{p_0})})\sum_{|\alpha|\le 2m}\lambda^{1-\frac{|\alpha|}{2m}}(|D^{\alpha}u|^{p_0})^{\frac{1}{p_0}}_{Q^{+}_{\kappa r}(X_0)}\\
&\quad +C\kappa^{\frac{d+2m}{p_0}}(|f|^{p_0})^{\frac{1}{p_0}}+C\kappa^{\frac{d+2m}{p_0}}\rho^{\frac{1}{p_0 \varsigma}}(|D^{2m}u|^{p_0\mu})^{\frac{1}{p_0\mu}}_{Q^{+}_{\kappa r}(X_0)}.
\end{aligned}
\end{equation*}
Observe that
\begin{multline*}
(u_{t})^{\sharp}(t,x)+\sum_{|\alpha|\le 2m, \alpha_1<2m}\lambda^{1-\frac{|\alpha|}{2m}}(D^{\alpha}u)^{\sharp}(t,x)
\le \sup(|u_{t}-(u_{t})_{Q^{+}_{r}(X_0)}|)_{Q^{+}_{r}(X_0)}\\
+\sup\sum_{\substack{|\alpha|\le 2m, \alpha_1<2m}}\lambda^{1-\frac{|\alpha|}{2m}}(|D^{\alpha}u-(D^{\alpha}u)_{Q^{+}_{r}(X_0)}|)_{Q^{+}_{r}(X_0)},
\end{multline*}
where the supremum is taken over all the $Q_r^+(X_0)$ with $(t,x)\in Q_{r}^{+}(X_0)$. This implies
\begin{equation}\label{eq:ineqbeforenorms}
\begin{aligned}
&(u_{t})^{\sharp}(t,x)+\sum_{\substack{|\alpha|\le 2m, \alpha_1<2m}}\lambda^{1-\frac{|\alpha|}{2m}}(D^{\alpha}u)^{\sharp}(t,x)\\
&\le C(\kappa^{2m(1-\frac{1}{p_0})}R_{1}^{2m(1-\frac{1}{p_0})}+\kappa^{-(1-\frac{1}{p_0})})\sum_{|\alpha|\le 2m}\lambda^{1-\frac{|\alpha|}{2m}}[\mathcal{M}(|D^{\alpha}u|^{p_0})(t,x)]^{\frac{1}{p_0}}\\
&\quad +C\kappa^{\frac{d+2m}{p_0}}[\mathcal{M}(|f|^{p_0})(t,x)]^{\frac{1}{p_0}}+C\kappa^{\frac{d+2m}{p_0}}\rho^{\frac{1}{p_0 \varsigma}}[\mathcal{M}(|D^{2m}u|^{p_0\mu})(t,x)]^{\frac{1}{p_0\mu}}.
\end{aligned}
\end{equation}
By taking the $L_{p,q,v,w}(\R^{d+1}_{+})$-norms on both sides of \eqref{eq:ineqbeforenorms} and applying Theorems \ref{thm:HL} and \ref{thm:ffst}, we get for $C=C(\theta,d,m,K,p,q,[v]_{p},[w]_{q},b_{j\beta})$,
\begin{equation}\label{eq:norms1}
\begin{aligned}
&\|u_{t}\|_{L_{p,q,v,w}(\R^{d+1}_{+})}+\sum_{\substack{|\alpha|\le 2m, \alpha_1<2m}}\lambda^{1-\frac{|\alpha|}{2m}}\|D^{\alpha}u\|_{L_{p,q,v,w}(\R^{d+1}_{+})}\\
&\le C\kappa^{\frac{d+2m}{p_0}}\|f\|_{L_{p,q,v,w}(\R^{d+1}_{+})}
+C\kappa^{\frac{d+2m}{p_0}}\rho^{\frac{1}{p_{0}\varsigma}}\|D^{2m}u\|_{L_{p,q,v,w}(\R^{d+1}_{+})}\\
&\quad +C(\kappa^{2m(1-\frac{1}{p_0})}R_{1}^{2m(1-\frac{1}{p_0})}+\kappa^{-(1-\frac{1}{p_0})})\sum_{|\alpha|\le 2m}\lambda^{1-\frac{|\alpha|}{2m}}\|D^{\alpha}u\|_{L_{p,q,v,w}(\R^{d+1}_{+})},
\end{aligned}
\end{equation}
where we used
\begin{equation*}
\begin{aligned}
&\|[\mathcal{M}(D^{2m}u)^{p_0\mu}]^{\frac{1}{p_0\mu}}\|_{L_{p,q,v,w}(\R^{d+1}_{+})}
=\|\mathcal{M}(D^{2m}u)^{p_0\mu}\|^{\frac{1}{p_0\mu}}_{L_{p/(p_0\mu),q/(p_0\mu),v,w}(\R^{d+1}_{+})}\\
&\le C\|(D^{2m}u)^{p_0\mu}\|^{\frac{1}{p_0\mu}}_{L_{p/(p_0\mu),q/(p_0\mu),v,w}(\R^{d+1}_{+})}
=C\|D^{2m}u\|_{L_{p,q,v,w}(\R^{d+1}_{+})},
\end{aligned}
\end{equation*}
with $C=C(d,p/(p_0 \mu),q/(p_0\mu),[v]_{p},[w]_{q})$.
		
It follows from the equation that
\[
a_{\tilde{\alpha}\tilde{\alpha}}(t,x)D^{2m}_{x_1}u=f-u_{t}-\sum_{|\alpha|= 2m,\alpha_{1}<2m}a_{\alpha}(t,x)D^{\alpha}u-\lambda u,
\]
where $\tilde{\alpha}=(m,0,\ldots,0)$.
Thus, by taking the $L_{p,q,v,w}$-norms and by the assumptions on the coefficients, it holds that for $C=C(\theta,d,m,K,p,q,[v]_{p},[w]_{q})$,
\begin{multline}\label{eq:norms2}
\|D^{2m}_{x_1}u\|_{L_{p,q,v,w}(\R^{d+1}_{+})}\le C\|f\|_{L_{p,q,v,w}(\R^{d+1}_{+})}+ C\|u_{t}\|_{L_{p,q,v,w}(\R^{d+1}_{+})} \\ +C\sum_{\substack{|\alpha|\le 2m, \alpha_1<2m}}\lambda^{1-\frac{|\alpha|}{2m}}\|D^{\alpha}u\|_{L_{p,q,v,w}(\R^{d+1}_{+})}.
\end{multline}
Combining \eqref{eq:norms1} and \eqref{eq:norms2}, we get
\begin{equation*}
\begin{aligned}
&\sum_{|\alpha|\le 2m}\lambda^{1-\frac{|\alpha|}{2m}}\|D^{\alpha}u\|_{L_{p,q,v,w}(\R^{d+1}_{+})}\\
&\le C\kappa^{\frac{d+2m}{p_0}}\|f\|_{L_{p,q,v,w}(\R^{d+1}_{+})}+C\kappa^{\frac{d+2m}{p_0}}\rho^{\frac{1}{p_{0}\varsigma}}\|D^{2m}u\|_{L_{p,q,v,w}(\R^{d+1}_{+})}\\
&\quad+ C(\kappa^{2m(1-\frac{1}{p_0})}R_{1}^{2m(1-\frac{1}{p_0})}+\kappa^{-(1-\frac{1}{p_0})})\sum_{|\alpha|\le 2m,\ \alpha_1\le 2m}\lambda^{1-\frac{|\alpha|}{2m}}\|D^{\alpha}u\|_{L_{p,q,v,w}(\R^{d+1}_{+})}.
\end{aligned}
\end{equation*}
Finally by first taking $\kappa\geq 16$ sufficiently large and then $\rho$ and $R_1$ sufficiently small such that
\[
C\kappa^{-(1-\frac{1}{p_{0}})}\le \frac{1}{6},\quad  C\kappa^{2m(1-\frac{1}{p_{0}})}R_{1}^{2m(1-\frac{1}{p_{0}})}\le\frac{1}{6},\quad
{\rm and}\quad C\kappa^{\frac{d+2m}{p_{0}}}\rho^{\frac{1}{p_{0}\varsigma}}\le \frac{1}{6},
\]
we get \eqref{eq:lemmanormspq}. The lemma is proved.
\end{proof}
		
From Lemma \ref{lemma:normspq} and using a partition of unity argument with respect to only the time variable,  we can prove the second intermediate result.
		
\begin{proposition}\label{prop:interpolation}
Assume that $A$ and $B_j$, $j=1,\ldots,m$, satisfy conditions \ref{as:operatorALS}, \ref{as:operatorBLS}, and \ref{as:LScond} for some $\theta\in(0,\pi/2)$, and assume the lower-order terms of $B_j$ to be all zero. Then there exists
$\rho=\rho(\theta,m,d,K,p,q,[v]_p,[w]_q,b_{j\beta})\in(0,1)$ such that for $\lambda\geq 0$, $f\in L_{p,q,v,w}(\overline{\R^{d+1}_{+}})$ and $u\in W^{1,2m}_{p,q,v,w}(\overline{\R^{d+1}_{+}})$ satisfying \eqref{prob:VMOtimespaceLS}, we have
\begin{align}\label{eq:propinterpolation}
&\sum_{|\alpha|\le 2m}\lambda^{1-\frac{|\alpha|}{2m}}\|D^{\alpha}u\|_{L_{p,q,v,w}(\R^{d+1}_{+})}\nonumber\\
&\le C_{1}\|f\|_{L_{p,q,v,w}(\R^{d+1}_{+})}
+C_2\sum_{|\alpha|\le 2m-1}\|D^{\alpha}u\|_{L_{p,q,v,w}(\R^{d+1}_{+})},
\end{align}
where
\begin{align*}
C_1&=C_1(\theta,d,m,K,p,q,[v]_{p},[w]_{q},b_{j\beta}),\\
C_2&=C_2(\theta,d,m,K,p,q,[v]_{p},[w]_{q},R_0,b_{j\beta}).
\end{align*}
\end{proposition}
\begin{proof}
Without loss of generality, we can assume the lower-order coefficients of $A$ to be zero.
To see this, just move the terms $\sum_{|\alpha|<2m}a_{\alpha}(t,x)D^{\alpha}$ to the right-hand side of \eqref{prob:VMOtimespaceLS}, i.e., consider
\[
u_{t}+\sum_{|\alpha|=2m}a_{\alpha}(t,x)D^{\alpha}u=f-\sum_{|\alpha|\le 2m-1}a_{\alpha}(t,x)D^{\alpha}u
\]
and recall that the lower-order coefficients of $A$ are bounded by $K$, so that
\[
\sum_{|\alpha|\le 2m-1}\|a_{\alpha}D^{\alpha}u\|_{L_{p,q,v,w}(\R^{d+1}_{+})}\le C_{K}\sum_{|\alpha|\le 2m-1}\|D^{\alpha}u\|_{L_{p,q,v,w}(\R^{d+1}_{+})}.
\]
If \eqref{eq:propinterpolation} holds for $A=\sum_{|\alpha|=2m}a_{\alpha}(t,x)D^{\alpha}$, we thus get
\begin{equation*}
\begin{aligned}
&\sum_{|\alpha|\le 2m}\lambda^{1-\frac{|\alpha|}{2m}}\|D^{\alpha}u\|_{L_{p,q,v,w}(\R^{d+1}_{+})}\\
&\le C_{1}\|f\|_{L_{p,q,v,w}(\R^{d+1}_{+})} + C_{1}C_{K}\sum_{|\alpha|\le 2m-1}\|D^{\alpha}u\|_{L_{p,q,v,w}(\R^{d+1}_{+})}\\
&\quad +C_2\sum_{|\alpha|\le 2m-1}\|D^{\alpha}u\|_{L_{p,q,v,w}(\R^{d+1}_{+})}\\
&\le C_{1}\|f\|_{L_{p,q,v,w}(\R^{d+1}_{+})}+C_2\sum_{|\alpha|\le 2m-1}\|D^{\alpha}u\|_{L_{p,q,v,w}(\R^{d+1}_{+})}.
\end{aligned}
\end{equation*}
Take now $R_1\in(0,1)$ from Lemma \ref{lemma:normspq} and fix a non-negative infinitely differentiable function $\zeta(t)$ defined on $\R$ such that $\zeta(t)$ vanishes outside $(-(R_0R_1)^{2m},0)$ and
\[
\int_{\R}\zeta(t)^{p}\,dt=1.
\]
Then, $u(t,x)\zeta(t-s)$ satisfies
\begin{equation}\label{eq:probpartofunity}
\begin{dcases}
(u(t,x)\zeta(t-s))_{t}+(\lambda+A)(u(t,x)\zeta(t-s))\\
\ \ =\zeta(t-s)f(t,x)+\zeta_{t}(t-s)u(t,x) & {\rm on}\ \R^{d+1}_{+}\\
B_{j}(u(t,x)\zeta(t-s))\big|_{x_{1}=0}=0 & {\rm on}\ \R\times \R^{d-1}.
\end{dcases}
\end{equation}
For each $s\in\R$, since $u(t,x)\zeta(t-s)$ vanishes outside $(s-(R_0R_1)^{2m},s)\times \R^{d}_{+}$, by Lemma \ref{lemma:normspq} applied to \eqref{eq:probpartofunity} we get
\begin{align}\label{eq:partofunity1}
&\sum_{|\alpha|<2m}\lambda^{1-\frac{|\alpha|}{2m}}
\|D^{\alpha}(u\zeta(\cdot-s))\|_{L_{p,q,v,w}(\R^{d+1}_{+})}\nonumber\\
&\le C\|f\zeta(\cdot-s)\|_{L_{p,q,v,w}(\R^{d+1}_{+})}
+C\|u\zeta_{t}(\cdot-s)\|_{L_{p,q,v,w}(\R^{d+1}_{+})},
\end{align}
where $C=C(d,m,K,p,q,[v]_{p},[w]_{q},b_{j\beta})$. Note that
\begin{equation*}
\begin{aligned}
\|D^{\alpha}u(t,\cdot)\|^{p}_{L_{q,w}(\R^{d}_{+})}&=\int_{\R}\|D^{\alpha}u(t,\cdot)\|^{p}_{L_{q,w}(\R^{d}_{+})}\zeta(t-s)^{p}\,ds\\
&=\int_{\R}\|D^{\alpha}u(t,\cdot)\zeta(t-s)\|^{p}_{L_{q,w}(\R^{d}_{+})}\,ds.
\end{aligned}
\end{equation*}
Thus, by integrating with respect to $t$,
\[
\|D^{\alpha}u\|^{p}_{L_{p,q,v,w}(\R^{d+1}_{+})}= \int_{\R}\|D^{\alpha}(u\zeta(\cdot-s))\|^{p}_{L_{p,q,v,w}(\R^{d+1}_{+})}\,ds.
\]
From this and \eqref{eq:partofunity1} it follows that
\[
\sum_{|\alpha|\le 2m}\lambda^{1-\frac{|\alpha|}{2m}}\|D^{\alpha}u\|_{L_{p,q,v,w}(\R^{d+1}_{+})}\le C_{1}\|f\|_{L_{p,q,v,w}(\R^{d+1}_{+})}+C_{2}\|u\|_{L_{p,q,v,w}(\R^{d+1}_{+})},
\]
where $C_1=C_1(\theta,d,m,K,p,q,[v]_p,[w]_q,b_{j\beta})>0$ and $C_2$ depends on $R_0R_1$ and the same parameters as $C_1$ does.
\end{proof}
		
Now Theorem \ref{thm:VMOproblemLS} follows from Proposition \ref{prop:interpolation}.
	
\begin{proof}[Proof of Theorem \ref{thm:VMOproblemLS}.] It suffices to consider $T=\infty$. For the general case when $T\in (-\infty,\infty]$, we can follow the proof of Lemma \ref{lemma:estconstcoeff} with the obvious changes in the weighted setting, so we omit the details.
	
(i) In Proposition \ref{prop:interpolation} we take
$\lambda_0\geq 0$ depending only on $C_2$ such that
\[
\frac{1}{2}\sum_{|\alpha|\le 2m-1}\lambda^{1-\frac{|\alpha|}{2m}}\le \sum_{|\alpha|\le 2m-1}\bigg(\lambda^{1-\frac{|\alpha|}{2m}}-C_2\bigg)
\]
for any $\lambda\geq\lambda_0$. By \eqref{eq:propinterpolation} we get
\begin{multline*}
\frac{1}{2}\sum_{|\alpha|\le 2m-1}\lambda^{1-\frac{|\alpha|}{2m}}\|D^{\alpha}u\|_{L_{p,q,v,w}(\R^{d+1}_{+})}
+\|D^{2m}u\|_{L_{p,q,v,w}(\R^{d+1}_{+})}\\
\le C\|f\|_{L_{p,q,v,w}(\R^{d+1}_{+})}
\end{multline*}
and thus
\begin{equation}\label{eq:propinterpfinal}
\sum_{|\alpha|\le 2m}\lambda^{1-\frac{|\alpha|}{2m}}\|D^{\alpha}u\|_{L_{p,q,v,w}(\R^{d+1}_{+})}\le C\|f\|_{L_{p,q,v,w}(\R^{d+1}_{+})}.
\end{equation}
Finally, the estimate of $\|u_{t}\|_{L_{p,q,v,w}(\R^{d+1}_{+})}$ follows by noting that $u_{t}=f-(A+\lambda)u$ and \eqref{eq:propinterpfinal}. This proves \eqref{eq:VMOtimespaceLS}.
			
(ii) As in the proof of Proposition \ref{prop:interpolation}, we can assume the lower-order coefficients of $A$ to be zero. Let
$$
A(0,0)u:=\sum_{|\alpha|=2m}a_{\alpha}(0,0)D^{\alpha}.
$$
By Lemma \ref{lemma:estconstcoeff}, we first solve
\begin{align*}
\begin{dcases}
\partial_t u_1+(\lambda+A(0,0))u_1=0 & {\rm in}\ \R^{d+1}_{+}\\
\sum_{|\beta|=m_j}b_{j\beta}D^{\beta}u_1\Big|_{x_{1}=0}=- \sum_{|\beta|<m_j}b_{j\beta}(t,x)D^{\beta}u\Big|_{x_{1}=0} + g_j& {\rm on}\ \R\times \R^{d-1},
\end{dcases}
\end{align*}
and by Theorem \ref{thm:spacialtracepq} we get
\begin{multline}\label{eq:probu1}
\|\partial_t u_1\|_{L_{p}(\R;L_{q}(\R^{d}_{+}))}+\sum_{|\alpha|\le 2m}\lambda^{1-\frac{|\alpha|}{2m}}\|D^{\alpha}u_1\|_{L_{p}(\R;L_{q}(\R^{d}_{+}))}\\
\le C\|\sum_{|\beta|<m_j}b_{j\beta}D^{\beta}u\|_{W^{(2m-m_j)\frac{1}{2m}}_{p}
	(\R;L_{q}(\R^{d}_{+}))\cap L_{p}(\R;W^{2m-m_j}_{q}(\R^{d}_{+}))}\\
+ \sum_{j=1}^m\|g_j\|_{F^{k_j}_{p,q}(\R;L_{q}(\R^{d-1}))\cap L_{p}(\R;\B_{q,q}^{2mk_j}(\R^{d-1}))}.
\end{multline}
Next $u_2=u-u_1$ satisfies the equation
\begin{equation*}
\begin{dcases}
\partial_t u_2+(\lambda+A)u_2=f-(A-A(0,0))u_1 & {\rm in}\ \R^{d+1}_{+}\\
\sum_{|\beta|=m_j}b_{j\beta}D^{\beta}u_2\Big|_{x_{1}=0}=0& {\rm on}\ \R\times \R^{d-1},
\end{dcases}
\end{equation*}
to which we can apply statement (i) with $v=w=1$ to get
\begin{equation}\label{eq:probu2}
\begin{aligned}
&\|\partial_t u_2\|_{L_{p}(\R;L_{q}(\R^{d}_{+}))}+\sum_{|\alpha|\le 2m}\lambda^{1-\frac{|\alpha|}{2m}}\|D^{\alpha}u_2\|_{L_{p}(\R;L_{q}(\R^{d}_{+}))} \\
&\le C\|f\|_{L_{p}(\R;L_{q}(\R^{d}_{+}))}+C\|(A-A(0,0))u_1\|_{L_{p}(\R;L_{q}(\R^{d}_{+}))}\\
&\le C\|f\|_{L_{p}(\R;L_{q}(\R^{d}_{+}))}+C_{K} \|D^{2m}u_1\|_{L_{p}(\R;L_{q}(\R^{d}_{+}))},
\end{aligned}
\end{equation}
with $\lambda\geq \lambda_0$, where $\lambda_0$ depends only on the constant $C_2$ from Proposition \ref{prop:interpolation}.
Now, since $u=u_1+u_2$, by \eqref{eq:probu2},
\begin{align*}
&\|u_t\|_{L_{p}(\R;L_{q}(\R^{d}_{+}))}+\sum_{|\alpha|\le 2m}\lambda^{1-\frac{|\alpha|}{2m}}\|D^{\alpha}u\|_{L_{p}(\R;L_{q}(\R^{d}_{+}))}\\
&\le \|\partial_t u_1\|_{L_{p}(\R;L_{q}(\R^{d}_{+}))}+\sum_{|\alpha|\le 2m}\lambda^{1-\frac{|\alpha|}{2m}}\|D^{\alpha}u_1\|_{L_{p}(\R;L_{q}(\R^{d}_{+}))}\\
&\quad +\|\partial_t u_2\|_{L_{p}(\R;L_{q}(\R^{d}_{+}))}+\sum_{|\alpha|\le 2m}\lambda^{1-\frac{|\alpha|}{2m}}\|D^{\alpha}u_2\|_{L_{p}(\R;L_{q}(\R^{d}_{+}))}\\
&\le \|\partial_t u_1\|_{L_{p}(\R;L_{q}(\R^{d}_{+}))}+\sum_{|\alpha|\le 2m}\lambda^{1-\frac{|\alpha|}{2m}}\|D^{\alpha}u_1\|_{L_{p}(\R;L_{q}(\R^{d}_{+}))}\\
&\quad +C_{K}\|D^{2m}u_1\|_{L_{p}(\R;L_{q}(\R^{d}_{+}))}
+C\|f\|_{L_{p}(\R;L_{q}(\R^{d}_{+}))},
\end{align*}
which by \eqref{eq:probu1} is further bounded by
\begin{align*}
&C\|f\|_{L_{p}(\R;L_{q}(\R^{d}_{+}))}
+C_K\sum_{j=1}^m\|g_j\|_{F^{k_j}_{p,q}(\R;L_{q}(\R^{d-1}))\cap L_{p}(\R;\B_{q,q}^{2mk_j}(\R^{d-1}))}\\
&\quad +
C_{K}\|\sum_{|\beta|<m_j}b_{j\beta}(t,x)
D^{\beta}u\|_{W^{(2m-m_j)\frac{1}{2m}}_{p}(\R;L_{q}(\R^{d}_{+}))\cap L_{p}(\R;W^{2m-m_j}_{q}(\R^{d}_{+}))}\\
&\le C\|f\|_{L_{p}(\R;L_{q}(\R^{d}_{+}))}+ C_K\sum_{j=1}^m\|g_j\|_{F^{k_j}_{p,q}(\R;L_{q}(\R^{d-1}))\cap L_{p}(\R;\B_{q,q}^{2mk_j}(\R^{d-1}))}\\
&\quad + C_{K}(C\varepsilon\|D^{2m}u\|_{L_{p}(\R;L_{q}(\R^{d}_{+}))}
+C\varepsilon\|u_t\|_{L_{p}(\R;L_{q}(\R^{d}_{+}))}+C_{\varepsilon}\|u\|_{L_{p}(\R;L_{q}(\R^{d}_{+}))}),
\end{align*}
where the last inequality follows from the smoothness the coefficients $b_{j\beta}(t,x)$ for $|\beta|<m_j$ and by using interpolation estimates as in Lemma \ref{lemma:interpolemmapq}. Now, taking $\varepsilon$ small enough so that $C_{K}C\varepsilon \le 1/2$ and $\lambda$ such that $\lambda\geq \max\{\lambda_0,2C_{K}C_{\varepsilon}\}$, we get \eqref{eq:VMOtimespaceLSgj}. \end{proof}
		
From Theorem \ref{thm:VMOproblemLS}, we now prove Theorem \ref{thm:VMOellipticLS}.
		
\begin{proof}[Proof of Theorem \ref{thm:VMOellipticLS}]
(i) Take $\zeta\in C_0^{\infty}(\R)$ and set $v(t,x)=\zeta(t/n)u(x)$, $n\in\Z$, which satisfies, in $\R^{d+1}_{+}$
\begin{equation}\label{prob:proofVMOellLS}
\begin{dcases}
v_{t}(t,x)+(A+\lambda)v(t,x)=h & {\rm in}\ \R\times\R^{d}_{+}\\
B_{j}v(t,x)\big|_{x_1=0}=0 & {\rm on}\ \R\times\R^{d-1},
\end{dcases}
\end{equation}
with $h:=\frac{1}{n}\zeta_{t}(\frac{t}{n})u(x)+\zeta(\frac{t}{n})f$.
If we now apply Theorem \ref{thm:VMOproblemLS} to \eqref{prob:proofVMOellLS} with $v=1$ we get
\begin{equation}\label{eq:proofVMOellLS}
\sum_{|\alpha|\le 2m}\lambda^{1-\frac{|\alpha|}{2m}}\|D^{\alpha}v\|_{L_{p}(\R;L_{q,w}(\R^{d}_{+}))}\le C\|h\|_{L_{p}(\R;L_{q,w}(\R^{d}_{+}))},
\end{equation}
with $C=C(\theta,m,d,K,p,q,R_0,[w]_{q},b_{j\beta})$. Observe now that
\begin{equation*}
\|h\|_{L_{p}(\R;L_{q,w}(\R^{d}_{+}))}\le \frac{1}{n}\|\zeta_t(\cdot/n)\|_{L_p(\R)}\|u\|_{L_{q,w}(\R^{d}_{+})}
+\|\zeta(\cdot/n)\|_{L_p(\R)}\|f\|_{L_{q,w}(\R^{d}_{+})},
\end{equation*}
and
\[
\|D^{\alpha}v\|_{L_{p}(\R;L_{q,w}(\R^{d}_{+}))}
=\|\zeta(\cdot/n)\|_{L_p(\R)}\|D^{\alpha}u\|_{L_{q,w}(\R^{d}_{+})}.
\]
Thus, combining the above estimates with \eqref{eq:proofVMOellLS} and letting $n\rightarrow +\infty$, we get \eqref{eq:VMOellipticLS}.
(ii) The estimate \eqref{eq:VMOellipticLSgj} follows in the same way from \eqref{eq:VMOtimespaceLSgj}.
\end{proof}

\section{Existence of solutions}\label{sec:solvab}
The a priori estimates of Theorems \ref{thm:VMOproblemLS} and \ref{thm:VMOellipticLS} can be used to derive the existence of solutions to the corresponding equations.
In this section we focus on the solvability of the parabolic problem \eqref{prob:VMOtimespaceLS}. The elliptic case follows in the same way from the a priori estimates in Theorem \ref{thm:VMOellipticLS}.

As in the proof of Lemma \ref{lemma:estconstcoeff}, via a standard argument it suffices to consider $T=\infty$. See, for instance, \cite[Theorem 2.1]{Kry07}.
Under the conditions in Theorem \ref{thm:VMOproblemLS}(ii), from the a priori estimate \eqref{eq:VMOtimespaceLSgj}, the standard method of continuity (see \cite[Theorem 5.2]{GilTru}) combined with Lemma \ref{lemma:estconstcoeff}, yields existence and uniqueness of a strong solution to \eqref{prob:VMOtimespaceLS}.

We now assume that the conditions in Theorem \ref{thm:VMOproblemLS}(i) are satisfied and we show the solvability of \eqref{prob:VMOtimespaceLS} via a density argument as in \cite[Section 8]{DK16}. By reverse H\"{o}lder's inequality and the doubling property of $A_p$-weights, one can find a sufficiently large constant $p_1$ and small constants $\varepsilon_1,\varepsilon_2\in(0,1)$ depending on $d$, $p$, $q$, $[v]_{p}$, $[w]_{q}$ such that
$$
1-\frac{p}{p_1}=\frac{1}{1+\varepsilon_1},\ \ 1-\frac{q}{p_1}=\frac{1}{1+\varepsilon_2},
$$
and both $v^{1+\varepsilon_1}$ and $w^{1+\varepsilon_2}$ are locally integrable and satisfy the doubling property, i.e. for every $r>0,\ t_0\in\R,\ x_0\in\R^{d}_{+}$,
\begin{align}                           \label{eq:veps1}
&\int_{I_{2r}(t_0)}v^{1+\varepsilon_1}dt\le C_0\int_{I_{r}(t_0)}v^{1+\varepsilon_1}dt,\\
                                \label{eq:weps1}
&\int_{B_{2r}^{+}(x_0)}w^{1+\varepsilon_1}dt\le C_0\int_{B_{r}^{+}(x_0)}w^{1+\varepsilon_1}dt,
\end{align}
where $C_0$ is independent of $r$, $t_0$, and $x_0$, and $I_{r}(t_0)=(t_0-r^{2m},t_0+r^{2m})$ denotes an interval in $\R$.
By H\"{o}lder's inequality, any function $f\in L_{p_1}(\R^{d+1}_{+})$ is locally in $L_{p,q,v,w}(\R^{d+1}_{+})$ and for any $r>0$,
\begin{equation}\label{eq:estweightp1}
\|f\|_{L_{p,q,v,w}(\mathcal{Q}_{r}^{+})}\le C\|f\|_{L_{p_1}(\mathcal{Q}_{r}^{+})},
\end{equation}
where $\mathcal{Q}_{r}^{+}=((-r^{2m},r^{2m})\times B_{r}) \cap \R^{d+1}_{+}$, with $B_{r}$ being a ball of radius $r$ in $\R^{d}$, and $C$ depends also on $r$.

Now if $f\in L_{p,q,v,w}(\R^{d+1}_{+})$, by the denseness of $C_{0}^{\infty}(\overline{\R^{d+1}_{+}})$ in $L_{p,q,v,w}(\R^{d+1}_{+})$, we can find a sequence of smooth functions $\{f_k\}_{k=0,1,\ldots}$ with bounded supports such that
\begin{equation}\label{eq:fkdense}
f_{k}\rightarrow f\ \ {\rm in}\ \ L_{p,q,v,w}(\R^{d+1}_{+})\ \ {\rm as}\ k\rightarrow\infty.
\end{equation}
Since for each $k$, $f_k\in L_{p_1}(\R^{d+1}_{+})$, by the solvability in the unweighted setting of Theorem \ref{thm:VMOproblemLS}(ii) with $p_1$ instead of $p=q$, zero lower-order coefficients for $B_j$ and $g_j\equiv 0$, there exists a unique solution $u_{k}\in W^{1,2m}_{p_1}(\R^{d+1}_{+})$ to
\begin{equation*}
\begin{dcases}
(u_k)_{t}(t,x) + (A+\lambda)u_k(t,x)=f_k(t,x) & {\rm in}\ \R\times\R^{d}_{+}\\
B_{j}u_k(t,x)\big|_{x_1=0}=0 & {\rm on}\ \R\times\R^{d-1},\ \ j=1,\ldots,m,
\end{dcases}
\end{equation*}
provided that $\lambda\geq \lambda_1(\theta,m,d,p_1,K,R_0,b_{j\beta})$ and  $\rho\le\rho_1(\theta,m,d,p_1,K,b_{j\beta})$.

We claim that if $\lambda\geq \max\{\lambda_0,\lambda_1\}$, then $u_k\in W^{1,2m}_{p,q,v,w}(\R^{d+1}_{+})$. If the claim is proved, it follows from the a priori estimate \eqref{eq:VMOtimespaceLS} and from \eqref{eq:fkdense} that $\{u_k\}$ is a Cauchy sequence in $W^{1,2m}_{p,q,v,w}(\R^{d+1}_{+})$. Let $u$ be its limit. Then, by taking the limit of the equation for $u_k$, it follows that $u$ is the solution to \eqref{prob:VMOtimespaceLS}.

In order to prove the claim, we fix a $k\in\N$ and we assume that $f_k$ is supported in $\mathcal{Q}_{R}^{+}$ for some $R\geq 1$. By \eqref{eq:estweightp1} we have
\begin{equation}\label{eq:proofsolv1}
\|D^{\alpha}u_k\|_{L_{p,q,v,w}(\mathcal{Q}_{2R}^{+})}<\infty,\ 0\le|\alpha|\le 2m
\end{equation}
and
\begin{equation}\label{eq:proofsolv2}
\|(u_k)_t\|_{L_{p,q,v,w}(\mathcal{Q}_{2R}^+)}<\infty.
\end{equation}
For $j\geq 0$, we take a sequence of smooth functions $\eta_j$ such that $\eta_j\equiv 0$ in $\mathcal{Q}_{2^j R}^{+}$, $\eta_j\equiv 1$ outside $\mathcal{Q}_{2^{j+1}R}^{+}$ and
\[
|D^{\alpha}\eta_j|\le C2^{-j|\alpha|},\ \ |\alpha|\le 2m,\ \ |(\eta_j)_{t}|\le C2^{-2mj}.
\]
Observe that $u_k\eta_j\in W^{1,2m}_{p_1}(\R^{d+1}_{+})$ satisfies
\begin{equation*}
\begin{dcases}
\partial_t(u_k\eta_j)+(A+\lambda)(u_k\eta_j)=f_j\ \ {\rm in}\ \R^{d+1}_{+}
\\B_j(u_k\eta_j)\big|_{x_1=0}={\rm tr}_{x_1=0}G_{j}\ \ {\rm on}\ \partial\R^{d+1}_{+},\  j=1,\ldots,m,
\end{dcases}
\end{equation*}
where by Leibnitz's rule
\[
f_{j}=u_k(\eta_{j})_{t}+\sum_{1\le |\alpha|\le 2m}\sum_{|\gamma|\le|\alpha|-1}\binom{\alpha}{\gamma}a_{\alpha}D^{\gamma}u_kD^{\alpha-\gamma}\eta_{j}\]
and
\[ G_{j}=\sum_{|\beta|=m_j}\sum_{|\tau|\le m_j-1}\binom{\beta}{\tau}b_{j\beta}D^{\tau}u_kD^{\beta-\tau}\eta_{j}.
\]

Now let
$$
g_j={\rm tr}_{x_1=0}G_j\in W_{p_1}^{1-\frac{m_j}{2m}-\frac{1}{2mp_1},2m-m_{j}-\frac{1}{p_1}}(\R\times \R^{d-1}).
$$
By applying the a priori estimate \eqref{eq:VMOtimespaceLSgj}, with $p_1$ instead of $p=q$ there, to $u_k\eta_j$, we get
\begin{multline*}
\|\partial_t(u_k\eta_j)\|_{L_{p_1}(\R^{d+1}_{+})}+
\sum_{|\alpha|\le 2m}\lambda^{1-\frac{|\alpha|}{2m}}\|D^{\alpha}(u_k\eta_j)\|_{L_{p_1}(\R^{d+1}_{+})}\\
\le C\|f_j\|_{L_{p_1}(\R^{d+1}_{+})} + C\sum_{j=1}^{m}\|g_{j}\|_{W_{p_1}^{1-\frac{m_j}{2m}-\frac{1}{2mp_1},2m-m_{j}-\frac{1}{p_1}}(\R\times \R^{d-1})},
\end{multline*}
with a constant $C=C(\theta,m,d,K,p_1,b_{j\beta})>0$. By Theorem \ref{thm:MeyLemma1.3.11} with $s=1-m_j/(2m)\in (0,1]$, $m_{j}\in\{0,\ldots,2m-1\}$, we have
\[
\|g_{j}\|_{W_{p_1}^{1-\frac{m_j}{2m}-\frac{1}{2mp_1},2m-m_{j}-\frac{1}{p_1}}(\R\times \R^{d-1})}\le C\|G_{j}\|_{W_{p_1}^{1-\frac{m_j}{2m},2m-m_{j}}(\R^{d+1}_{+})}.
\]
Observe that
\[
\|f_j\|_{L_{p_1}(\R^{d+1}_{+})}\le C \|(\eta_{j})_{t}u_k\|_{L_{p_1}(\R^{d+1}_{+})}+C\sum_{1\le |\alpha|\le 2m}
\sum_{|\gamma|\le |\alpha|-1}\|D^{\gamma}u_k D^{\alpha-\gamma}\eta_{j}\|_{L_{p_1}(\R^{d+1}_{+})}
\]
and
\begin{align*}
\|G_{j}\|_{W_{p_1}^{1-\frac{m_j}{2m},2m-m_{j}}(\R^{d+1}_{+})}
\le C\sum_{|\beta|=m_j}\sum_{|\tau|\le m_j-1}\|D^{\tau}u_k D^{\beta-\tau}\eta_j\|_{W_{p_1}^{1-\frac{m_j}{2m},2m-m_{j}}(\R^{d+1}_{+})}.
\end{align*}
This implies that
\begin{multline*}
\|\partial_t(u_k\eta_j)\|_{L_{p_1}(\R^{d+1}_{+})}+\sum_{|\alpha|\le 2m}\lambda^{1-\frac{|\alpha|}{2m}}\|D^{\alpha}(u_k\eta_j)\|_{L_{p_1}(\R^{d+1}_{+})}\\
\le C \|(\eta_{j})_{t}u_k\|_{L_{p_1}(\R^{d+1}_{+})}
+C\sum_{1\le |\alpha|\le 2m}
\sum_{|\gamma|\le |\alpha|-1}\|D^{\gamma}u_k D^{\alpha-\gamma}\eta_{j}\|_{L_{p_1}(\R^{d+1}_{+})}\\ + C\sum_{|\beta|=m_j}\sum_{|\tau|\le m_j-1}\|D^{\tau}u_k D^{\beta-\tau}\eta_j\|_{W_{p_1}^{1-\frac{m_j}{2m},2m-m_{j}}(\R^{d+1}_{+})},
\end{multline*}
from which it follows that
\begin{multline*}
\|(u_k)_{t}\|_{L_{p_1}(\R^{d+1}_{+}\backslash\mathcal{Q}_{2^{j+1}R}^{+})}+\sum_{|\alpha|\le 2m}\lambda^{1-\frac{|\alpha|}{2m}}\|D^{\alpha}u_k\|_{L_{p_1}(\R^{d+1}_{+}\backslash\mathcal{Q}^{+}_{2^{j+1}R})}\\
\le C2^{-j} \|u_k\|_{L_{p_1}(\mathcal{Q}^{+}_{2^{j+1}R}\backslash\mathcal{Q}^{+}_{2^{j}R})}
+C2^{-j}\sum_{1\le |\alpha|\le 2m}
\sum_{|\gamma|\le |\alpha|-1}\|D^{\gamma}u_k\|_{L_{p_1}(\mathcal{Q}^{+}_{2^{j+1}R}\backslash\mathcal{Q}^{+}_{2^{j}R})}\\ + C2^{-j}\sum_{|\beta|=m_j}\sum_{|\tau|\le m_j-1}\|D^{\tau}u_k\|_{W_{p_1}^{1-\frac{m_j}{2m},2m-m_{j}}(\mathcal{Q}^{+}_{2^{j+1}R}\backslash\mathcal{Q}^{+}_{2^{j}R})}.
\end{multline*}
By standard interpolation inequalities (see e.g. \cite{krylov}),
\[
\|D^{\gamma}u_k\|_{L_{p}(\mathcal{Q}^{+}_{2^{j+1}R}
\setminus\mathcal{Q}^{+}_{2^{j}R})}\le
C\|D^{2m}u_k\|_{L_{p}(\mathcal{Q}^{+}_{2^{j+1}R}\backslash\mathcal{Q}^{+}_{2^{j}R})}
+C\|u_k\|_{L_{p}(\mathcal{Q}^{+}_{2^{j+1}R}\backslash\mathcal{Q}^{+}_{2^{j}R})},
\]
and by the interpolation estimates as in Lemma \ref{lemma:MeyLemma1.3.13},
\begin{align*}
&\|D^{\tau}u_k\|_{W_{p}^{1-\frac{m_j}{2m},2m-m_{j}}(\mathcal{Q}^{+}_{2^{j+1}R}\backslash\mathcal{Q}^{+}_{2^{j}R})}\\
&\le C\|D^{2m}u_k\|_{L_{p}(\mathcal{Q}^{+}_{2^{j+1}R}\backslash\mathcal{Q}^{+}_{2^{j}R})}
+C\|(u_k)_t\|_{L_{p}(\mathcal{Q}^{+}_{2^{j+1}R}
\backslash\mathcal{Q}^{+}_{2^{j}R})}+C\|u_k\|_{L_{p}(\mathcal{Q}^{+}_{2^{j+1}R}\backslash\mathcal{Q}^{+}_{2^{j}R})}.
\end{align*}

Thus, we get
\begin{multline*}
\|(u_k)_{t}\|_{L_{p_1}(\mathcal{Q}^{+}_{2^{j+2}R}\backslash\mathcal{Q}^{+}_{2^{j+1}R})}+\sum_{|\alpha|\le2m}\lambda^{1-\frac{|\alpha|}{2m}}\|D^{\alpha}u_k\|_{L_{p_1}(\mathcal{Q}^{+}_{2^{j+2}R}\backslash\mathcal{Q}^{+}_{2^{j+1}R})}\\
\le C2^{-j}(\|(u_k)_{t}\|_{L_{p_1}(\mathcal{Q}^{+}_{2^{j+1}R}\backslash\mathcal{Q}^{+}_{2^{j}R})}+\|D^{2m}u_k\|_{L_{p_1}(\mathcal{Q}^{+}_{2^{j+1}R}\backslash\mathcal{Q}^{+}_{2^{j}R})}\\
+\|u_k\|_{L_{p_1}(\mathcal{Q}^{+}_{2^{j+1}R}\backslash\mathcal{Q}^{+}_{2^{j}R})}).
\end{multline*}
By induction, we obtain for each $j\geq 1$,
\begin{multline}\label{eq:inductest}
\|(u_k)_{t}\|_{L_{p_1}(\mathcal{Q}^{+}_{2^{j+1}R}\backslash\mathcal{Q}^{+}_{2^{j}R})}+\sum_{|\alpha|\le 2m}\lambda^{1-\frac{|\alpha|}{2m}}\|D^{\alpha}u_k\|_{L_{p_1}(\mathcal{Q}^{+}_{2^{j+1}R}\backslash\mathcal{Q}^{+}_{2^{j}R})}\\
\le C^j 2^{-\frac{j(j-1)}{2}}(\|(u_k)_{t}\|_{L_{p_1}(\mathcal{Q}^{+}_{2R})}+\|D^{2m}u_k\|_{L_{p_1}(\mathcal{Q}^{+}_{2R})}+\|u_k\|_{L_{p_1}(\mathcal{Q}^{+}_{2R})}).
\end{multline}
Finally, by Holder's inequality, \eqref{eq:veps1}, \eqref{eq:weps1} and \eqref{eq:inductest}, we get for each $j\geq 1$,
\begin{equation*}
\begin{split}
&\|(u_k)_{t}\|_{L_{p,q,v,w}(\mathcal{Q}^{+}_{2^{j+1}R}\backslash\mathcal{Q}^{+}_{2^{j}R})}+\sum_{|\alpha|\le 2m}\lambda^{1-\frac{|\alpha|}{2m}}\|D^{\alpha}u_k\|_{L_{p,q,v,w}(\mathcal{Q}^{+}_{2^{j+2}R}\backslash\mathcal{Q}^{+}_{2^{j+1}R})}\\
&\le \|v\|^{\frac{1}{p}}_{L_{1+\varepsilon_1}(I_{2^{j+1}R})}\|w\|^{\frac{1}{q}}_{L_{1+\varepsilon_2}(B^+_{2^{j+1}R})}
\Big(\|(u_k)_t\|_{L_{p_1}(\mathcal{Q}^{+}_{2^{j+2}R}\backslash\mathcal{Q}^{+}_{2^{j+1}R})}\\
&\quad +\|D^{2m}u_k\|_{L_{p_1}(\mathcal{Q}^{+}_{2^{j+1}R}\backslash\mathcal{Q}^{+}_{2^{j}R})}+\|u_k\|_{L_{p_1}(\mathcal{Q}^{+}_{2^{j+1}R}
\backslash\mathcal{Q}^{+}_{2^{j}R})}\Big)\\
&\le CC^{j(1+\frac{1}{p}+\frac{1}{q})}2^{-\frac{j(j-1)}{2}}
\Big(\|(u_k)_t\|_{L_{p_1}(\mathcal{Q}^{+}_{2R})}+
\|D^{2m}u_k\|_{L_{p_1}(\mathcal{Q}^{+}_{2R})}+\|u_k\|_{L_{p_1}(\mathcal{Q}^{+}_{2R})}\Big).
\end{split}
\end{equation*}
The above inequality together with \eqref{eq:proofsolv1} and \eqref{eq:proofsolv2} implies that $u_k\in W^{1,2m}_{p,q,v,w}(\R^{d+1}_{+})$, which proves the claim.

\begin{remark}
Under certain compatibility condition, the solvability of the corresponding initial-boundary value problem can also be obtained. See, for instance, \cite[Sect. 2.5]{krylov} and \cite{DHP07} for details.
\end{remark}

\def\cprime{$'$}

\end{document}